\documentclass[a4paper,12pt]{amsart}
\usepackage[latin1]{inputenc}
\usepackage[francais]{babel}
\usepackage[T1]{fontenc}
\usepackage{latexsym}
\usepackage{amsmath,amsthm}
\usepackage{amsfonts}
\usepackage{epsfig}
\usepackage{amssymb}
\usepackage{amscd}
\usepackage[all]{xy}
\usepackage{stmaryrd}
\usepackage{rotating}
\input xy

\usepackage{hyperref}

\newtheorem{Lemma}{Lemme}[section]
\newtheorem{Theo}[Lemma]{Théorème}

\newtheorem{Pro}[Lemma]{Proposition}

\theoremstyle{definition}
\newtheorem{Def}[Lemma]{Définition}
\newtheorem{Example}[Lemma]{Exemple}

\theoremstyle{definition}

\theoremstyle{remark}

\theoremstyle{remark}
\newtheorem{Rem}[Lemma]{Remarque}

\newcommand{\NN}{\mathbb{N}}
\newcommand{\RR}{\mathbb{R}}
\newcommand{\CC}{\mathbb{C}}
\newcommand{\ZZ}{\mathbb{Z}}

\newcommand{\G}{\rtimes^{\rho} G}
\newcommand{\rG}{\rtimes^{\rho}_{r} G}

\newcommand{\End}{\mathrm{End}}
\numberwithin{equation}{section}
\newcommand{\E}{\overline{E}}
\newcommand{\A}{\mathcal{A}^{\rho}_r(G)}

\newcommand{\Amax}{\mathcal{A}^{\rho}(G)}

\newcommand{\EEnd}{\mathrm{End}}

\title[Morphisme de Baum-Connes tordu]{Morphisme de Baum-Connes tordu par une représentation non unitaire}

\author{Maria Paula Gomez-Aparicio}

\address{Maria Paula {\sc Gomez-Aparicio} : Institut de Mathématiques de Jussieu, Projet d'algèbres d'Opérateurs et représentations,
175 rue du chevaleret, 75013 Paris, France.}

\email{gomez@math.jussieu.fr}

\keywords{Non-unitary representations, Banach algebras, Baum-Connes conjecture}
\subjclass[2000]{22D12, 22D15, 46L80, 19K35}

\date{}

\begin{document}
\maketitle

\begin{abstract}
Soit $G$ un groupe localement compact et $\rho$ une représentation de dimension finie de $G$ non unitaire. On définit des algèbres de Banach analogues aux $C^*$-algèbres de groupe, $C^*(G)$ et $C^*_r(G)$, en considérant l'ensemble des représentations de la forme $\rho\otimes\pi$, où $\pi$ parcourt un ensemble de représentations unitaires de $G$. On calcule la $K$-théorie de ces algèbres pour une large classe de groupes vérifiant la conjecture de Baum-Connes.\\

\vspace{2mm}
\noindent{\sc {\large a}bstract.} Let $G$ be a locally compact group and $\rho$ a non-unitary finite dimensional representation of $G$. We consider tensor
products of $\rho$ by some unitary representations of $G$ in order to define two Banach algebras analogous to the group $C^*$-algebras, $C^*(G)$ and $C^*_r(G)$. We calculate the $K$-theory of such algebras for a large class of groups satisfying the Baum-Connes conjecture.

\end{abstract}


\tableofcontents

\section*{Introduction}
Soit $G$ un groupe localement compact et $\rho$ une représentation de dimension finie de $G$. Dans \cite{Gomez07}, nous avons défini un analogue tordu par $\rho$ de la $C^*$-algèbre maximale de $G$, que l'on note $\mathcal{A}^{\rho}(G)$, en considérant la complétion de l'espace des fonctions continues à support compact sur $G$, que l'on note $C_c(G)$, pour la norme
$$\|f\|=\sup\limits_{(\pi,H)}\|(\pi\otimes\rho)(f)\|_{\mathcal{L}(H\otimes V)},$$
pour $f\in C_c(G)$ et où le supremum est pris parmi les représentations unitaires de $G$. Ces algèbres de groupe \emph{tordues} sont des algèbres de Banach; ce sont des $C^*$-algèbres si et seulement si $\rho$ est unitaire. Elles apparaissaient alors de façon très naturelle dans l'étude du comportement de $\rho$ dans l'ensemble des représentations de $G$ de la forme $\rho\otimes\pi$, avec $\pi$ unitaire. En effet, nous avons alors défini un renforcement de la propriété (T) de Kazhdan \cite{Kazhdan} en termes d'idempotents dans $\Amax$ qui nous a permis de montrer que, pour la plupart des groupes de Lie semi-simples réels ayant la propriété (T), toute représentation irréductible de dimension finie $\rho$ est isolée dans l'ensemble des représentations de la forme $\rho\otimes\pi$, où $\pi$ parcourt l'ensemble des représentations unitaires et  irréductibles de $G$.\\

 D'autre part, dans le même article, nous avons aussi défini un analogue tordu de la $C^*$-algèbre réduite de $G$, noté $\mathcal{A}_r^{\rho}(G)$, en considérant la norme sur $C_c(G)$ donnée par la formule
$$\|f\|=\|(\lambda_{G}\otimes\rho)(f)\|_{\mathcal{L}(L^{2}(G)\otimes V)},$$
où $\lambda_{G}$ est la représentation régulière gauche de $G$. Nous avons alors montré que si le groupe $G$ est non-compact et a la propriété (T) tordue définie dans \cite{Gomez07}, ces deux algèbres tordues n'ont pas la même $K$-théorie. Lorsque $\rho$ est unitaire, ceci est un résultat classique: dans ce cas les algèbres tordues coïncident avec les $C^*$-algèbres de groupe $C^*(G)$ et $C^*_r(G)$, respectivement, et la propriété (T) tordue coïncide avec la propriété (T) de Kazhdan. C'est un résultat connu qui dit que si un groupe non-compact $G$ a la propriété (T), alors $C^*(G)$ et $C^*_r(G)$ n'ont pas la même $K$-théorie, ce qui a d'ailleurs constitué pendant longtemps une barrière pour la vérification de la conjecture de Baum-Connes pour des groupes infinis discrets ayant la propriété (T) (cf. \cite{Julg97}).\\

Le but de cet article est de calculer la $K$-théorie des algèbres tordues pour une large classe de groupes vérifiant la conjecture de Baum-Connes.  Pour ceci, nous allons construire deux morphismes \emph{tordus}, $\mu_{\rho}$ et $\mu_{\rho,r}$, qui vont du membre de gauche du morphisme de Baum-Connes dans la $K$-théorie des algèbres tordues et qui coïncident avec les morphismes de Baum-Connes classiques, $\mu$ et $\mu_r$, si $\rho$ est unitaire. Les algèbres tordues étant des algèbres de Banach, notre outil principal sera la $KK$-théorie banachique de Lafforgue (cf. \cite{Lafforgue02}). Nous allons alors montrer que les algèbres tordues se comportent de la même façon que $C^*(G)$ et $C^*_r(G)$ au niveau de la $K$-théorie.\\

 On rappelle que la conjecture de Baum-Connes propose une façon de calculer la $K$-théorie de $C^*(G)$ pour tout groupe localement compact $G$ (cf. \cite{Baum-Connes-Higson}). Plus précisément, Baum, Connes et Higson ont défini un morphisme d'assemblage $$\mu_r:K^{\mathrm{top}}(G)\rightarrow K(C_r^*(G)),$$
 où $K^{\mathrm{top}}(G)$ est la $K$-homologie $G$-équivariante à support compact du classifiant universel pour les actions propres de $G$, noté $\underline{E}G$. Ce morphisme, appelé désormais \emph{morphisme de Baum-Connes}, peut être défini à l'aide de la $KK$-théorie équivariante de Kasparov (cf. \cite{Kasparov88}). La conjecture de Baum-Connes affirme que $\mu_r$ est un isomorphisme pour tout groupe localement compact $G$.\\

 La méthode la plus puissante pour montrer la conjecture de Baum-Connes, appelée de façon générale méthode du ``dual Dirac-Dirac'' a été introduite par Kasparov dans son preprint de 1981 (publié après dans \cite{Kasparov88}) pour démontrer la conjecture de Novikov dans le cas des variétés dont le groupe fondamental est un sous-groupe discret d'un groupe de Lie connexe. Elle a été ensuite énoncée dans une forme très générale par Tu (cf. \cite{Tu99}) qui consiste à construire un élément de ``Dirac'' $d$ dans $KK_G(A,\CC)$ et un élément ``dual-Dirac'' $\eta$ dans $KK_G(\CC,A)$, pour $A$ une $G$-$C^*$-algèbre propre, tels que, si on considère l'élément de $KK_G(\CC,\CC)$ défini par la formule $\gamma:=\eta\otimes_A d$, où $\otimes_A$ denote le produit de Kasparov au-dessus de $A$, alors $\gamma$ doit agir par l'identité sur $K^{\mathrm{top}}(G)$; plus précisément, on demande que $p^*(\gamma)=1$ dans $KK_{G\ltimes \underline{E}G}(C_0(\underline{E}G),C_0(\underline{E}G))$, où $p:\underline{E}G\rightarrow \{pt\}$ est la projection de $\underline{E}G$ sur le point. Un élément $\gamma$ avec ces propriétés est appelé ``élément $\gamma$ de Kasparov''. Tu a montré que si un élément $\gamma$ de Kasparov existe, alors le morphisme de Baum-Connes $\mu_r$ est injectif. Si de plus $\gamma=1$ dans $KK_G(\CC,\CC)$, alors $\mu_r$ est un isomorphisme.

Par ailleurs, on peut aussi construire un morphisme $$\mu:K^{\mathrm{top}}(G)\rightarrow K(C^*(G)),$$
(cf. \cite{Baum-Connes-Higson}). Les résultats de Tu impliquent que s'il existe un élément $\gamma=1$ dans $KK_G(\CC,\CC)$ comme ci-dessus, alors $\mu$ est aussi un isomorphisme (cf. \cite{Tu99}).\\

 Dans \cite{Kasparov88}, Kasparov a utilisé la méthode originale pour montrer l'injectivité de $\mu_r$ (et donc la conjecture de Novikov) pour tout groupe de Lie semi-simple et pour tout sous-groupe fermé d'un groupe de Lie semi-simple. Depuis, un élément $\gamma$ de Kasparov a été construit, par exemple, par Kasparov et Skandalis et puis par Higson et Kasparov, pour une classe très vaste de groupes, notée $\mathcal{C}$ dans \cite{Lafforgue02}. Nous rappelons ici, par souci de commodité pour le lecteur, que cette classe est constituée par les groupes suivants:\\

\begin{itemize}
 \item les groupes localement compacts agissant de façon continue, propre et isométrique sur une variété riemannienne complète simplement connexe à courbure sectionnelle négative ou nulle (cf. \cite{Kasparov88}), ou sur un immeuble de Bruhat-Tits affine (cf. \cite{Kasparov-Skandalis91}),
\item les groupes discrets agissant proprement et par isométries sur un espace métrique faiblement géodésique, faiblement bolique et de géométrie co-uniforme bornée (cf. \cite{Kasparov-Skandalis03} et \cite{Tzanev00} pour la terminologie co-uniforme),
\item les groupes localement compacts a-T-menables, c'est-à-dire qui agissent de façon affine, isométrique et propre sur un espace de Hilbert (cf. \cite{Higson-Kasparov}).\\
\end{itemize}
La classe $\mathcal{C}$ contient, en particulier, tous les groupes moyennables, tous les groupes hyperboliques au sens de Gromov et tous les groupes $p$-adiques.\\
Dans \cite{Julg-Kasparov}, Julg et Kasparov ont aussi prouvé l'égalité $\gamma=1$ dans $KK_G(\CC,\CC)$, et donc la bijectivité du morphisme de Baum-Connes, pour $SU(n,1)$. Higson et Kasparov ont ensuite généralisé leur résultat pour tous les groupes a-T-menables (cf. \cite{Higson-Kasparov}).\\

Revenons maintenant aux algèbres de groupe tordues. Pour tout groupe localement compact (et dénombrable à l'infini) et pour toute représentation $\rho$ de dimension finie, nous allons construire deux morphismes
$$\mu_{\rho}:K^{\mathrm{top}}(G)\rightarrow K(\Amax)\quad\text{et}\quad\mu_{\rho,r}:K^{\mathrm{top}}(G)\rightarrow K(\A).$$
Nous allons ensuite montrer que $\mu_{\rho}$ et $\mu_{\rho,r}$ sont des isomorphismes pour tout groupe localement compact pour lequel il existe un élément $\gamma$ de Kasparov qui est égal à $1$ dans $KK_G(\CC,\CC)$. Plus précisément, nous allons montrer les deux théorèmes suivants. On rappelle que $$K^{\mathrm{top}}(G)=\lim\limits_{\longrightarrow}KK_G(C_0(Y),\CC),$$où la limite inductive est prise parmi les parties $Y$ $G$-compactes de $\underline{E}G$. 
\begin{Theo}
\sloppy Supposons que pour toute partie $G$-compacte $Y$ de $\underline{E}G$, il existe une $G$-$C^*$-algèbre propre $B$ et $\eta\in KK_G(\CC,B)$ et $d\in KK_G(B,\CC)$ tels que $\gamma=\eta\otimes_B d\in KK_G(\CC,\CC)$ vérifie $p^*(\gamma)=1$ dans $KK_{G\ltimes Y}(C_0(Y),C_0(Y))$, où $p$ est la projection de $Y$ vers le point, si bien que $\gamma$ agisse par l'identité sur $K^{\mathrm{top}}(G)$. Alors, pour toute représentation $\rho$ de dimension finie, les morphismes $\mu_{\rho}$ et $\mu_{\rho,r}$ sont injectifs.
\end{Theo}

\begin{Theo}
Soit $G$ un groupe localement compact tel qu'il existe une $G$-$C^*$-algèbre propre $B$, et des éléments $\eta\in KK_G(\CC,B)$ et $d\in KK_G(B,\CC)$
 tels que si on pose $\gamma=\eta\otimes_B d\in KK_G(\CC,\CC)$ on a $\gamma=1$. Alors, pour toute
 représentation $\rho$ de dimension finie de $G$, $\mu_{\rho}$ et $\mu_{\rho,r}$ sont des isomorphismes.
\end{Theo}

Ceci implique en particulier que $\mu_{\rho}$ et $\mu_{\rho,r}$ sont injectifs pour tout groupe appartenant à la classe $\mathcal{C}$ et ce sont des isomorphismes pour tout groupe a-T-menable. Dans ce cas, les algèbres $\Amax$, $\A$, $C^*(G)$ et $C^*_r(G)$ ont donc toutes la même $K$-théorie.\\

Dans un autre article (cf. \cite{Gomez08-2}), qui fait partie de \cite{GomezThese}, nous montrerons que $\mu_{\rho,r}$ est un isomorphisme pour tout groupe ayant la propriété (RD) et appartenant à une sous-classe de $\mathcal{C}$, notée $\mathcal{C'}$ par Lafforgue (cf. \cite[Introduction]{Lafforgue02}). En particulier, on obtiendra la bijectivité du morphisme $\mu_{\rho,r}$ pour tout groupe de Lie réductif réel, pour tout groupe hyperbolique et pour tous les sous-groupes discrets et cocompacts de $SL_3(F)$, où $F$ est un corps local, de $SL_3(\mathbb{H})$ et de $E_{6(-26)}$ (cf. \cite{Lafforgue00}, \cite{Chatterji03}). Le morphisme de Baum-Connes tordu $\mu_{\rho,r}$ est alors un isomorphisme pour la plupart des groupes pour lesquels on sait montrer que le morphisme de Baum-Connes classique l'est. De plus, la bijectivité du morphisme tordu ne semble pas être plus facile à démontrer que la conjecture de Baum-Connes elle-même, l'algèbre tordue $\A$, à différence des complétions inconditionnelles introduites par Lafforgue, n'étant pas plus stable que $C^*_r(G)$ par le produit de Schur (\cite{LafforgueICM02}). Cependant, on va montrer (cf. propostion \ref{petites}) que les algèbres tordues peuvent être très ``petites'', c'est-à-dire contenues dans des algèbres $L^1$ qui sont des complétions inconditionnelles (cf. \cite{Lafforgue02}). Ceci nous fait croire que la conjecture de Baum-Connes est fortement liée à la bijectivité du morphisme tordu et on peut espérer que les deux soient vérifiées toujours au même temps. On rappelle que Higson, Lafforgue et Skandalis dans \cite{Higson-Lafforgue-Skandalis} ont donné un contre-exemple à la généralisation de la conjecture de Baum-Connes aux actions de groupe (connue comme la conjecture de Baum-Connes à coefficients), ce qui nous laisse penser que le morphisme tordu ne doit pas être un isomorphisme pour tous les groupes localement compacts; mais on peut espèrer que ça soit le cas pour tous les groupes de la classe $\mathcal{C}$.\\
Dans le cas des groupes abéliens, les algèbres tordues avaient déjà étés considérées par Bost dans le cadre du principe d'Oka (cf. \cite{Bost90}).

\medskip
\noindent{\bf Remerciements.} Ce travail fait partie des travaux présentés pour l'obtention de mon Doctorat réalisé sous la direction de Vincent Lafforgue. Je tiens à le remercier pour sa grande disponibilité et ses suggestions. Je remercie aussi Georges Skandalis pour ses éclaircissements et ses conseils et Hervé Oyono-Oyono pour ses commentaires.

\section{Algèbres de groupe tordues}\label{produitcroisé}

\subsection{Définitions et propriétés principales}
Soit $G$ un groupe localement compact et soit $dg$ une mesure de Haar à gauche sur $G$. On note $\Delta$ la fonction modulaire de $G$ (c'est-à-dire que $dg^{-1}=\Delta(g)^{-1}dg$ pour tout $g\in G$).\\
 Soit $A$ une $G$-$C^*$-algèbre. Pour tout $g\in G$ et pour tout $a\in A$, on note $g.a$, ou $g(a)$, l'action de $g$ sur $a$. On considère alors l'espace vectoriel des fonctions continues à support compact sur $G$ et à valeurs dans $A$, que l'on note $C_c(G,A)$, muni de la structure d'algèbre involutive dont la multiplication est donnée par
$$(f_1*f_2)(g)=\int_Gf_1(g_1)g_1(f_2(g_1^{-1}g))dg_1,$$ pour $f_1,f_2\in C_c(G,A)$ et l'involution par
$$f^*(g)=g(f(g^{-1}))^*\Delta(g^{-1}),$$ pour $f\in C_c(G,A)$ et $g\in G$.
 De façon générale, on représente tout élément $f$ de $C_c(G,A)$ par l'intégrale formelle $\int_Gf(g)e_gdg$, où $e_g$ est une lettre formelle satisfaisant le conditions suivantes
$$e_ge_{g'}=e_{gg'},\quad e_g^*=(e_g)^{-1}=e_{g^{-1}}\quad\text{et}\quad e_gae_g^*=g.a,$$ pour tous $g,g'\in G$ et pour tout $a\in A$.

On note $C^*(G,A)$ et $C^*_r(G,A)$ les produits croisés, maximal et réduit respectivement, de $G$ et $A$. De plus, on note  
$$ L^2(G,A)=\{f\in C_c(G,A)|\int_G f(g)^*f(g)dg\,\,\text{converge dans}\, A\},$$
et $\lambda_{G,A}$ la représentation régulière gauche de $C_c(G,A)$ dans $L^2(G,A)$
donnée par la formule $$\lambda_{G,A}(f)(h)(t)=\int_Gt^{-1}(f(s))h(s^{-1}t)ds,$$ pour $f\in C_c(G,A)$, $h\in L^2(G,A)$ et $t\in G$. On rappelle que $\lambda_{G,A}$ induit un unique morphisme de $C^*$-algèbres de $C^*(G,A)$ dans $C^*_r(G,A)$ que l'on note encore $\lambda_{G,A}$ par abus de notation.\\

 Tout le long de l'article, une représentation $\rho$ de dimension finie de $G$ sera une représentation de $G$ sur un espace vectoriel complexe de dimension finie, muni d'une structure hermitienne. On note $(\rho,V)$ toute représentation de $G$ sur un espace $V$ pour simplifier les notations et quand on veut préciser l'espace sur lequel $G$ agit. Le produit tensoriel de $C^*$-algèbres sera toujours le produit tensoriel minimal.\\
 
 Soit $(\rho,V)$ une représentation de dimension finie de $G$. On considère alors l'application
\begin{align*}
 C_c(G,A)&\rightarrow C_c(G,A)\otimes\End(V)\\
\int_Gf(g)e_gdg&\mapsto\int_Gf(g)e_g\otimes\rho(g)dg.
\end{align*}
Elle nous permet de donner la définition de \emph{produits croisés tordus} suivante. Soit $C^*(G,A)\otimes \End(V)$ (resp.$C^*_r(G,A)\otimes\End(V)$) le produit tensoriel minimal de $C^*$-algèbres.
\begin{Def}
Le \emph{produit croisé} \emph{tordu par $\rho$} (resp. \emph{produit croisé tordu réduit}), noté $A\rtimes^{\rho} G$ (resp. $A\rtimes^{\rho}_r G$), est le complété-séparé de $C_c(G,A)$
pour la norme
$$\|\int_Gf(g)e_gdg\|_{A\G}=\|\int_Gf(g)e_g\otimes\rho(g)dg\|_{C^*(G,A)\otimes\End(V)},$$
(resp. $\|.\|_{C^*_r(G,A)\otimes\End(V)}$).\\
Si $A=\CC$, alors on note $$\mathcal{A}^{\rho}(G):=\CC\rtimes^{\rho}G\quad\hbox{et}\quad\mathcal{A}^{\rho}_r(G):=\CC\rtimes_r^{\rho}G.$$
\end{Def}

On a alors la proposition suivante

\begin{Pro}
 Les produits croisés tordus $A\G$ et $A\rG$ sont des algèbres de Banach. Ce sont des $C^*$-algèbres si et seulement si $\rho$ est une représentation unitaire. Dans ce cas, $A\rtimes^{\rho}G=C^*(G,A)$ et  $A\rtimes_r^{\rho}G=C_r^*(G,A)$, à équivalence de norme près.
\end{Pro}
\begin{proof}
Il est clair que ce sont des algèbres de Banach. On va montrer qu'elles coïncident avec les produits croisés $C^*$-algébriques si et seulement si $\rho$ est unitaire dans le cas où $A=\CC$, le cas général étant analogue. Supposons que $\rho$ soit une représentation unitaire de $G$. Par définition, si $f\in C_c(G)$ alors $$\|f\|_{\Amax}=\sup\limits_{(\pi,H)}\|(\pi\otimes\rho)(f)\|_{\mathcal{L}(H\otimes V)},$$
où le supremum est pris parmi les représentations unitaires de $G$. \\
Alors on a trivialement l'inégalité de normes $\|.\|_{\Amax}\leq\|.\|_{C^*(G)}$ de sorte que $\Amax$ est un quotient de $C^*(G)$. Soit $(\rho^*,V^*)$ la représentation contragrédiente de $G$ sur l'espace dual $V^*$ de $V$. Donc, comme $(V^*\otimes V)^G=\mathrm{Hom}_{G}(V,V)$, la représentation triviale $1_G$ de $G$ est fortement contenue dans $\rho^*\otimes\rho$. Ceci implique que toute représentation unitaire $\pi$ est fortement contenue dans $\pi\otimes\rho^*\otimes\rho$, et donc que toute représentation unitaire $\pi$ est fortement contenue dans l'ensemble $\{\sigma\otimes\rho |\,\,\sigma\,\,\hbox{une représentation unitaire}\}$. D'où $\|.\|_{C^*(G)}\leq\|.\|_{\Amax}$ et $C^*(G)=\Amax$, à équivalence de norme près.\\
D'autre part, pour $f\in C_c(G)$,
$$\|f\|_{\A}=\|(\lambda_G\otimes\rho)(f)\|_{\mathcal{L}(L^2(G)\otimes V)},$$
où $\lambda_G:G\rightarrow\mathcal{L}(L^2(G))$ est la représentation régulière gauche de $G$. Dans ce cas, le résultat vient du fait que la représentation régulière de $G$ est ``absorbante'': la représentation $\lambda_G\otimes\rho$ est unitairement équivalente à $\lambda_G\otimes \mathrm{Id}_G$, où on note $\mathrm{Id}_G$ la représentation triviale $G$ sur $V$; l'opérateur d'entrelacement entre ces deux représentations est donné par l'application
\begin{align*}
 L^2(G,V) &\rightarrow L^2(G,V)\\
f&\mapsto \big(g\mapsto f(g)\rho(g^{-1})\big),
\end{align*}
\sloppy (en identifiant $L^2(G)\otimes V$ avec $L^2(G,V)$). Il est facile de vérifier que, $T(\lambda_G\otimes\rho)(g)=(\mathrm{Id}_G\otimes\lambda_G)(g)T$,  pour tout $g\in G$.\\
\end{proof}

\begin{Rem}
\begin{enumerate}
\item Il est clair que $\lambda_{G,A}$ induit un unique morphisme d'algèbres de Banach
$$\lambda_{G,A}^{\rho}:A\G\to A\rG.$$
\item Si on choisit une autre norme sur $V$, comme deux normes sur $V$ sont toujours équivalentes, on obtient alors une norme équivalente sur $A\G$.
En particulier, si $G$ est un groupe compact, comme toute représentation de $G$ sur un espace de Hilbert est unitarisable, alors $A\G=C^*(G,A)$, à équivalence de norme près. De même dans le cas de $A\rG$.
\item Soit $\rho^*$ la représentation contragrédiente de $\rho$ sur l'espace dual $V^*$ de $V$. Si $\rho$ et $\rho^*$ sont conjuguées par un opérateur unitaire de $V$ dans $V^*$, alors $\Amax$ et $\A$ sont des algèbres involutives.
\end{enumerate}
\end{Rem}

\begin{Example}
 Soit $G=\ZZ$ et soit $\rho:\ZZ\to \CC$ un caractère de $\ZZ$. Soit $S^{\rho}$ le cercle dans $\CC$ de rayon $|\rho(1)|$. Alors $\A$ est l'algèbre des fonctions continues sur $S^{\rho}$.
\end{Example}

La proposition suivante montre que les algèbres tordues peuvent être ``petites'', c'est-à-dire contenues dans des algèbres $L^1$.
\begin{Pro}\label{petites}
 Soit $\Gamma$ est un groupe discret et $\rho$ une représentation de $\Gamma$ sur un espace vectoriel de dimension finie $V$ muni d'une norme hermitienne et telle que $\sum\limits_{\gamma}\frac{1}{\|\rho(\gamma)\|_{\mathrm{End}(V)}}$ converge. Alors, pour toute $\Gamma$-$C^*$-algèbre $A$,
$$A\rtimes^{\rho}_r(\Gamma)\subset l^1(\Gamma,A)\subset C^*_r(\Gamma,A).$$
\end{Pro}
\begin{proof}
 Soit $A$ une $\Gamma$-$C^*$-algèbre. Supposons d'abord par simplicité que $A$ est unifère et notons $1_A$ son unité. Soit $\lambda_{\Gamma,A}$ la représentation régulière gauche de $C^*_r(\Gamma,A)$ dans $l^2(\Gamma,A)$. On note $\delta$ l'élément de $l^2(\Gamma,A)$ qui envoie l'identité $e$ de $\Gamma$ vers $1_A$ et qui est nulle sur $\gamma\neq e$. Soit $f\in C_c(\Gamma,A)$ que l'on note sous la forme $\sum\limits_{\gamma}f(\gamma)e_{\gamma}$. On a alors que $$\|\lambda_{\Gamma,A}(f)\delta\|_{l^2(\Gamma,A)}\leq\|\sum\limits_{\gamma}f(\gamma)e_{\gamma}\|_{C^*_r(\Gamma,A)},$$
par défintion de $C^*_r(\Gamma,A)$.
Or,
\begin{align*}
 \|\sum\limits_{\gamma}f(\gamma)e_{\gamma}\|_{l^2(\Gamma,A)}&=\|\sum\limits_{\gamma}\gamma^{-1}\big(f(\gamma)^*f(\gamma)\big)\|^{\frac{1}{2}}_A\\
&=\|\lambda_{\Gamma,A}(f)\delta\|_{l^2(\Gamma,A)},
\end{align*}
donc, l'application
\begin{align*}
 \phi:C_c(\Gamma,A)&\rightarrow C_c(\Gamma,A)\\
\sum\limits_{\gamma}f(\gamma)e_{\gamma}&\mapsto\sum\limits_{\gamma}e_{\gamma}\gamma^{-1}(f(\gamma)),
\end{align*}
se prolonge en une application continue injective de $C^*_r(\Gamma,A)$ dans $l^2(\Gamma,A)$ de norme inférieure ou égale à $1$.\\
De plus,
\begin{align*}
 \|\sum\limits_{\gamma}\gamma^{-1}\big(f(\gamma)^*f(\gamma)\big)\|^{\frac{1}{2}}_A&\geq\sup\limits_{\gamma\in\Gamma}\|{\gamma}^{-1}\big(f(\gamma)^*f(\gamma)\big)\|^{\frac{1}{2}}_A\\
&\geq \sup\limits_{\gamma\in\Gamma}\|f(\gamma)^*f(\gamma)\|^{\frac{1}{2}}_A\\
&\geq \sup\limits_{\gamma\in\Gamma}\|f(\gamma)\|_A,
\end{align*}
donc $\|f\|_{l^{\infty}(\Gamma,A)}\leq\|f\|_{l^2(\Gamma,A)}$, où $l^{\infty}(\Gamma,A)$ est le complété de $C_c(\Gamma,A)$ pour la norme $$\|f\|=\sup\limits_{\gamma\in\Gamma}\|f(\gamma)\|_A,$$ pour $f\in C_c(\Gamma,A)$. Donc, $\phi$ se prolonge en une application injective de $C^*_r(\Gamma,A)$ dans $l^{\infty}(\Gamma,A)$ qui est de norme inférieure ou égale à $1$.\\
Soit $l^{\infty,\rho}(\Gamma,A)$ le complété de $C_c(\Gamma,A)$ pour la norme $$\|f\|=\sup\limits_{\gamma\in\Gamma}\|f(\gamma)\|_A\|\rho(\gamma)\|_{\mathrm{End}(V)}.$$
Comme $A\rtimes^{\rho}_r\Gamma$ et $l^{\infty,\rho}(\Gamma,A)$ s'envoient de façon isométrique dans $C^*_r(\Gamma,A)\otimes\mathrm{End}(V)$ et dans $l^{\infty}(\Gamma,A)\otimes\mathrm{End}(V)$ respectivement, on a que pour toute $f\in C_c(\Gamma,A)$
$$\|f\|_{l^{\infty,\rho}(\Gamma,A)}\leq \|f\|_{A\rtimes^{\rho}_r\Gamma}.$$
Soit $C$ une constante positive telle que $\sum\limits_{\gamma}\frac{1}{\|\rho(\gamma)\|}<C$. On a alors, pour tout $f\in C_c(\Gamma,A)$,
\begin{align*}
 \|f\|_{l^1(\Gamma,A)}&=\sum\limits_{\gamma}\|f(\gamma)\|_A\|\rho(\gamma)\|\frac{1}{\|\rho(\gamma)\|}\\
&\leq\big(\sup\limits_{\gamma\in\Gamma}\|f(\gamma)\|_A\|\rho(\gamma)\|\big)\sum\limits_{\gamma}\frac{1}{\|\rho(\gamma)\|}\\
&\leq C\|f\|_{l^{\infty,\rho}(\Gamma,A)}\\
&\leq C\|f\|_{A\rtimes^{\rho}_r\Gamma},
\end{align*}
donc, il existe une application continue de norme inférieure ou égale à $1$ de $A\rtimes^{\rho}_r\Gamma$ dans $l^1(\Gamma,A)$ qui prolonge l'indentité sur $C_c(\Gamma,A)$.\\

Supposons maintenant que $A$ n'est pas unifère. Soit $(u_i)_{i\in I}$ une unité approchée de $A$. Pour tout $i\in I$, soit $\delta_i$ la fonction sur $\Gamma$ qui envoie l'identité de $\Gamma$ sur $u_i$ et telle que, pour tout élément $\gamma$ de $\Gamma$ différent de l'identité, $\delta_i(\gamma)$ est nul. Dans ce cas, pour tout $f\in C_c(\Gamma,A)$,
$$\lim\limits_i\|\lambda_{\Gamma,A}(f)\delta_i\|_A=\|f\|_{l^2(\Gamma,A)},$$
et comme $\|\delta_i\|_{l^2(\Gamma,A)}\leq 1$, car $\|u_i\|_A\leq 1$,
$$\|f\|_{l^2(\Gamma,A)}\leq\|f\|_{C^*_r(\Gamma,A)},$$
ce qui implique que $C^*_r(\Gamma,A)\subset l^2(\Gamma,A)$. On alors que $C^*_r(\Gamma,A)\subset l^{\infty}(\Gamma,A)$ et la même démonstration que dans le cas unifère montre que s'il existe une constante $C>0$ telle que $\sum\limits_{\gamma}\frac{1}{\|\rho(\gamma)\|}<C$, alors $$A\rtimes^{\rho}_r\Gamma\subset l^1(\Gamma,A).$$

\end{proof}

\subsection{Fonctorialité} Le lemme suivant dit que la construction des produits croisés tordus est fonctorielle.
Nous donnons d'abord la définition suivante
\begin{Def}
 Soient $B$ et $C$ deux $G$-$C^*$-algèbres $\rho$ une représentation de dimension finie de $G$ et $\theta:B\rightarrow C$ un morphisme $G$-équivariant de $C^*$-algèbres. On note $\tilde{\theta}$ l'application linéaire continue de $C_c(G,B)$ dans $C_c(G,C)$ telle que pour tout $f\in C_c(G,B)$, $$\tilde{\theta}(f)(g)=\theta(f(g)).$$
\end{Def}

\begin{Lemma}
 L'application $\tilde{\theta}$ se prolonge en un morphisme d'algèbres de Banach $\theta\G:B\G\rightarrow C\G$ (resp. $\theta\rG:B\rG\rightarrow C\rG$).
\end{Lemma}

\begin{proof}
 En effet,
\begin{align*}
 \|(\theta\G)(f)\|_{C\G}&=\|\int_G\theta(f(g))e_g\otimes\rho(g)dg\|_{C^*(G,C)\otimes \mathrm{End}(V)}\\
\leq\|(C^*(G,\theta&)\otimes\mathrm{Id}_V)\int_Gf(g)e_g\otimes\rho(g)dg\|_{C^*(G,C)\otimes \mathrm{End}(V)}\\
\leq\|C^*(G,\theta)&\otimes\mathrm{Id}_V\|_{\mathrm{Hom}(C^*(G,B)\otimes\mathrm{End}(V),C^*(G,C)\otimes\mathrm{End}(V))}\|f\|_{B\G},
\end{align*}
où $C^*(G,\theta):C^*(G,B)\rightarrow C^*(G,C)$ est le morphisme induit par $\theta$. Autrement dit, on a un diagramme commutatif
$$\xymatrix{
    C_c(G,B)\ar[r]\ar[d]^{\tilde{\theta}} &B\G\ar[r]& C^*(G,B)\otimes\EEnd(V)\ar[d]^{C^*(G,\theta)\otimes\mathrm{Id}_V}\\
C_c(G,C)\ar[r]&C\G\ar[r] & C^*(G,C)\otimes\EEnd(V).
}$$
Il en est de même dans le cas du produit croisé tordu réduit.
\end{proof}

\section{Morphisme de Baum-Connes tordu}
Notre but est de calculer la $K$-théorie des produits croisés tordus. 
Soit $\underline{E}G$ l'espace classifiant pour les actions propres de $G$. Pour toute $G$-$C^*$-algèbre $A$, on note $K^{\mathrm{top}}(G,A)$ la $K$-homologie $G$-équivariante de $\underline{E}G$ à valeurs dans $A$ introduite dans \cite{Baum-Connes-Higson}. Dans cette section, pour toute représentation $\rho$ de dimension finie de $G$, nous allons construire deux morphismes de groupes:
$$\mu^A_{\rho}:K^{\mathrm{top}}(G,A)\rightarrow K(A\G)\quad\hbox{et}\quad\mu^A_{\rho,r}:K^{\mathrm{top}}(G,A)\rightarrow K(A\rG).$$

\subsection{Flèche de descente tordue}\label{Defdescente}

Soient G un groupe localement compact et $(\rho,V)$ une
représentation de dimension finie de $G$. Soient $A$ et $B$ deux $G$-$C^*$-algèbres. Dans \cite[Theorem 3.11]{Kasparov88}, Kasparov a défini deux morphismes de descente, $j^G$ et $j^G_r$, qui vont de $KK_G(A,B)$ dans $KK(C^*(G,A),C^*(G,B))$ et dans $KK(C_r^*(G,A),C_r^*(G,B))$, respectivement, et qui permettent de définir le morphisme de Baum-Connes en utilisant la $KK$-théorie de Kasparov. Nous allons ici définir deux morphismes de descente \emph{tordus}
\begin{align*}
 j_{\rho}&:KK_G(A,B)\rightarrow KK^{\mathrm{ban}}(A\G,B\G),\\
j_{\rho,r}&:KK_G(A,B)\rightarrow KK^{\mathrm{ban}}(A\rG,B\rG),
\end{align*}
analogues à $j^G$ et $j^G_r$. 
On remarque cependant que, comme $A\G$ et $B\G$ sont des algèbres de Banach, ces morphismes ont nécessairement une image dans la $KK$-théorie banachique de Lafforgue, notée $KK^{\mathrm{ban}}$, et non pas dans la $KK$-théorie de Kasparov. La construction des morphismes tordus est analogue à la construction de la descente banachique dans le cas des complétions inconditionnelles \cite[Section 1.5]{Lafforgue02}.\\

\noindent{\bf Notations.} Introduisons d'abord quelques notations. Si $A$ est une $C^*$-algèbre, on note $\tilde{A}$ son algèbre unitarisée.

On appelle longueur sur $G$ toute fonction $\ell:G\rightarrow [0,+\infty[$ continue et telle que $\ell(g_1g_2)\leq\ell(g_1)+\ell(g_2)$, pour tout $g_1,g_2\in G$.\\

Soient $A$ et $B$ deux
$G$-$C^*$ algèbres et soit $E$ un $G$-$(A,B)$-bimodule de Kasparov
(c'est-à-dire que $E$ est un $G$-$B$-module hilbertien et $A$ agit sur
$E$ par un morphisme $G$-équivariant de $A$ dans $\mathcal{L}(E)$ cf. \cite{Kasparov88}). On note $<.,.>:E\times E\rightarrow B$ la forme sesquilinéaire qui fait de $E$ un $B$-module hilbertien. La norme sur $E$ définie par $$\|x\|_E=\|\langle x,x\rangle\|_B^{\frac{1}{2}},$$ fait de $E$ un espace de Banach sur lequel $G$ agit de façon continue. De plus, on a que $\|xb\|_E\leq\|x\|_E\|b\|_B$ et $\|gx\|_E\leq\|x\|_E$ pour tout $g\in G$, $b\in B$ et $x\in E$, donc $E$ est un $G$-$B$-module de Banach à droite (cf. \cite[Section 1.1]{Lafforgue02}).\\

\sloppy On considère le $G$-$B$-module de Banach à gauche non-dégénéré déterminé par $E$, que l'on note $\overline{E}$, tel qu'il existe une isométrie $\CC$-antilinéaire $*:E\rightarrow\overline{E}$ vérifiant $b^*x^*=(xb)^*$ pour $x\in E$ et $b\in B$. On définit alors un crochet $\langle.,.\rangle:\overline{E}\times E\rightarrow B$, que l'on note $\langle.,.\rangle$ par abus de notation, par la formule $\langle x^*,y\rangle=\langle x,y\rangle$ pour $x,y\in E$.

D'après \cite[propostion 1.14]{Lafforgue02}, $(\overline{E},E)$ est alors une $G$-$B$-paire et l'action $G$-équivariante de $A$ sur $E$ fait de $(\overline{E},E)$ un $G$-$(A,B)$-bimodule de Banach. Si $T\in\mathcal{L}_B(E)$ l'application $\overline{T}^{*}=(*)T^*(*)^{-1}$ définit un élément de $\mathcal{L}_B(\overline{E})$.\\

Pour toute longueur $\ell$ sur le groupe $G$, on note $\iota$ l'application:
$$\iota:KK_G(A,B)\rightarrow KK_{G,\ell}^{\mathrm{ban}}(A,B),$$
définie dans \cite[Section 1.6]{Lafforgue02} et déterminée par l'application:
\begin{align*}
 E_G(A,B)&\rightarrow E^{\mathrm{ban}}_{G,\ell}(A,B)\\
(E,T)&\mapsto \big((\overline{E},E),(\overline{T}^{*},T)\big),
\end{align*}
où on note $E_G(A,B)$ l'ensemble des cycles $G$-équivariants sur le couple $(A,B)$ (cf. \cite{Kasparov88}, voir aussi \cite[Defintion 9.4]{Skandalis91}) et $E^{\mathrm{ban}}_{G,\ell}(A,B)$ est l'ensemble des classes d'isomorphisme de cycles banachiques $(G,\ell)$-équivariants sur $(A,B)$ (cf. \cite[Définition 1.2.2]{Lafforgue02}).\\

Si $B$ est une algèbre de Banach, $E$ est un $B$-module de Banach à droite et $F$ est un $B$-module de Banach à gauche, on note $E\otimes^{\pi}_BF$ le produit tensoriel projectif de $E$ et $F$ au-dessus de $B$, c'est-à-dire le complété-séparé du produit tensoriel algébrique $E\otimes_{\CC}^{alg}F$ pour la plus grande semi-norme $\|.\|$ telle que $\|x\otimes by-xb\otimes y\|=0$ et $\|x\otimes y\|\leq\|x\|\,\|y\|$ pour $x\in E$, $y\in F$ et $b\in B$.\\

Si $E$ et $F$ sont deux $B$-paires, on note $\mathcal{L}_B(E,F)$ l'espace des morphismes de $B$-paires. Pour $\xi\in E^<$ et $y\in F^>$, on rappelle que l'on note $|y\rangle\langle\xi|\in\mathcal{L}_B(E,F)$ le morphisme de $B$-paires défini par
   $$\begin{array}{rl}
   |y\rangle\langle\xi|^>:E^>&\rightarrow F^>\\
x&\mapsto y\langle\xi,x\rangle,
\end{array}\quad\text{et}\quad
\begin{array}{rl}
|y\rangle\langle\xi|^<:F^<&\rightarrow E^<\\
\eta&\mapsto \langle\eta,y\rangle\xi.
  \end{array}$$
Un morphisme de $B$-paires de $E$ dans $F$ est compact s'il est limite dans $\mathcal{L}_B(E,F)$ de combinaisons linéaires de morphismes de la forme $|y\rangle\langle x|$ (cf. \cite[Définition 1.1.3]{Lafforgue02}). On note $\mathcal{K}_B(E,F)$ l'espace des morphismes compacts de $E$ dans $F$.\\

 On considère l'espace vectoriel $C_c(G,E)$ (resp. $C_c(G,\overline{E})$) des fonctions continues à support compact sur $G$ à valeurs dans $E$ (resp. à valeurs dans $\overline{E}$) et on note $x\in C_c(G,E)$ sous la forme
$x=\int_Gx(g)e_gdg$ (resp. $\xi\in C_c(G,\overline{E})$ sous la forme
$\xi=\int_Ge_g\xi(g)dg$).\\

 \sloppy Soient $C^*(G,E)$ et $C^*_r(G,E)$ les complétés de $C_c(G,E)$ définis comme dans \cite[Definition 3.8]{Kasparov88}. On considère alors le module hilbertien $C^*(G,E)\otimes\End(V)$ défini sur le produit tensoriel de $C^*$-algèbres $C^*(G,B)\otimes\mathrm{End}(V)$ et construit par produit tensoriel externe. De même, soit $C_r^*(G,E)\otimes\End(V)$ le $\big(C_r^*(G,B)\otimes\mathrm{End}(V)\big)$-module hilbertien.

\begin{Def}
On note $E\G$ (resp. $E\rG$) l'adhérence de l'image de $C_c(G,E)$ dans $C^*(G,E)\otimes \mathrm{End}(V)$ (resp. $C^*_r(G,E)\otimes \mathrm{End}(V)$) par l'application
\begin{align*}
 C_c(G,E)&\rightarrow C_c(G,E)\otimes\End(V)\\
\int_Gx(g)e_gdg&\mapsto\int_Gx(g)e_g\otimes\rho(g)dg.
\end{align*}
De même, on note $\overline{E}\G$ (resp. $\overline{E}\rG$) l'adhérence de l'image de $C_c(G,\overline{E})$ dans $\overline{C^*(G,E)}\otimes \mathrm{End}(V)$ (resp. $\overline{C^*_r(G,E)}\otimes \mathrm{End}(V)$) par l'application
\begin{align*}
C_c(G,\overline{E})&\rightarrow C_c(G,\overline{E})\otimes\End(V)\\
\int_Ge_g\xi(g)dg&\mapsto\int_Ge_g\xi(g)\otimes\rho(g)dg.
\end{align*}
\end{Def}

Nous allons maintenant définir une structure de $B\G$-paire (resp. $B\rG$-paire) sur le couple $(\overline{E}\G,E\G)$ (resp. $(\overline{E}\rG,E\rG)$) que l'on note $E\G$ (resp. $E\rG$) par abus de notation. On donne alors la définition suivante, qui est complètement analogue à la définition \cite[Définition 1.5.3]{Lafforgue02}
\begin{Def}
Soient $x=\int_Gx(g)e_gdg$ dans $C_c(G,E)$, $\xi=\int_Ge_g\xi(g)dg$ dans $C_c(G,\E)$ et $b=\int_Gb(g)e_gdg$ dans $C_c(G,B)$.\\ On pose
\begin{align*}
 x.b&=\int_G\int_Gx(t)t(b(t^{-1}g))dt e_gdg,\\
 b.\xi&=\int_Ge_g\int_Gg^{-1}(b(t))\xi(t^{-1}g)dtdg,\\
 \langle\xi,x\rangle&=\int_G\int_Gt(\langle\xi(t),x(t^{-1}g)\rangle) dt e_gdg.
\end{align*}
Ceci définit une structure de $B\G$-paire (resp. $B\rG$-paire) sur $(\E\G,E\G)$ (resp. $(\E\rG,E\rG)$).
\end{Def}
Maintenant, comme on a supposé de plus que $E$ est muni d'une structure de $A$-$B$-bimodule hilbertien, pour $A$ une $G$-$C^*$-algèbre, on va montrer la proposition suivante
\begin{Pro}
La paire $E\G$ (resp. $E\rG$) est un $(A\G,B\G)$-bimodule (resp. $(A\rG,B\rG)$-bimodule) de Banach.
\end{Pro}
\begin{proof}
Soit $a=\int_Ga(g)e_gdg\in C_c(G,A)$. L'algèbre $C_c(G,A)$ agit sur $C_c(G,E)$ de la façon suivante
\begin{align*}
 a.x&=\int_G\int_Ga(t)t(x(t^{-1}g))dte_gdg,\\
\xi.a&=\int_Ge_g\int_G(t^{-1}g)^{-1}(\xi(t)a(t^{-1}g))dtdg.
\end{align*}
On doit montrer que, avec ces formules, $E\G$ est un $A\G$-module de Banach à gauche, $\E\G$ un $A\G$-module de Banach à droite, que ces structures commutent avec les structures de $B\G$-modules qui découlent de la définition précédente et que, pour tout élément $a$ de $A\G$, $$\langle\xi a,x\rangle=\langle\xi,ax\rangle,$$ pour tout $\xi\in \E\G$ et pour tout $x\in E\G$. Mais ceci découle immédiatement du fait que
$\left( \begin{array}{cc}
        A\G & E\G\\
    \E\G & B\G
       \end{array}\right )$ est un sous-espace de
$$\begin{array}{cc}
       \left ( \begin{array}{cc}
          C^*(G,A) & C^*(G,E)\\
    \overline{C^*(G,E)} &  C^*(G,B)
         \end{array} \right ) \otimes&\mathrm{End}(V)\\
\end{array}$$
et que l'inclusion est une isométrie.\\
Le même raisonnement montre l'énoncé pour les produits croisés réduits.

\end{proof}

Considérons maintenant un élément de $E_G(A,B)$ que l'on note $(E,T)$. Soit $$T\G:E\G\rightarrow E\G,$$ (resp. $T\rG:E\rG\rightarrow E\rG$) le morphisme de $B\G$-paires (resp. $B\rG$-paires) défini sur $x\in C_c(G,E)$ et sur $\xi \in C_c(G,\E)$ de la façon suivante
\begin{align*}
(T\G)^{>}x(g)&= T^>(x(g)),\\
 (T\G)^{<}\xi(g)&=T^<(\xi(g)),
\end{align*}
(resp. $T\rG$).
Alors, $$\|T\G\|_{\mathcal{L}(E\G)}\leq\|T\|_{\mathcal{L}(E)},$$
(de même pour le produit croisé réduit). Ces opérateurs sont analogues aux opérateurs $\mathcal{A}(G,T)$ et $\tilde{T}$ définis dans \cite[Définition 1.5.3]{Lafforgue02} et \cite[Theorem 3.11]{Kasparov88} respectivement.
\begin{Lemma}\label{bimodtordu}
L'élément $(E\G,T\G)$ (resp. $(E\rG,T\rG)$) ainsi défini appartient à $E^{\mathrm{ban}}(A\G,B\G)$ (resp. $E^{\mathrm{ban}}(A\rG,B\rG)$).
\end{Lemma}
\begin{proof}
 On va montrer le lemme dans le cas du produit croisé maximal, le cas réduit étant complètement analogue.\\ On doit montrer que, pour tout élément $a\in C_c(G,A)$ et pour tout $g\in G$, les opérateurs $$[a,T\G],\quad a(1-(T\G)^2)\quad\text{et}\quad a(g(T\G)-T\G),$$ sont des opérateurs compacts de $E\G$.\\
\sloppy On remarque d'abord que l'opérateur $|y\rangle\langle\eta|\in\mathcal{K}_{B\G}(E\G)$, pour $\eta\in C_c(G,\overline{E})$ et $y\in C_c(G,E)$, agit sur $x\in C_c(G,E)$ et $\xi\in C_c(G,\overline{E})$ par les formules
\begin{align*}
 |y\rangle\langle\eta|^>(x)(g)&=\int_G K_s^>(s(x(s^{-1}g)))ds,\\
|y\rangle\langle\eta|^<(\xi)(g)&=\int_G (s^{-1}g)^{-1}(K_{s^{-1}g}^<(\xi(s)))ds,
\end{align*}
où $K_g=\int_G|y(s)\rangle\langle g(\eta(s^{-1}g))|ds$, pour tout $g\in G$, appartient à $\mathcal{K}_B(E)$.
\begin{Def}\label{Stilda}
On note $E$ la $B$-paire $(\overline{E},E)$ par abus de notation. Étant donné un élément $S=(S_g)_{g\in G}$ appartenant à $C_c(G,\mathcal{K}_B(E))$, on définit un opérateur $\widehat{S}$ dans $\mathcal{L}_{B\G}(E\rtimes^{\rho} G)$ de la manière suivante
$$\widehat{S}^>(x)(g)=\int_GS_s^>(s(x(s^{-1}g)))ds,$$
$$\widehat{S}^<(\xi)(g)=\int_G(s^{-1}g)^{-1}(S_{s^{-1}g}^<(\xi(s)))ds,$$
pour $x\in C_c(G,E)$ et $\xi\in
C_c(G,\E)$.
\end{Def}

\begin{Pro}\label{morBancompacts}
L'application
\begin{align*}
\psi:C_c(G,\mathcal{K}_B(E))&\rightarrow\mathcal{L}_{B\G}(E\G)\\
S&\mapsto \widehat{S},
\end{align*}
induit un morphisme d'algèbres de Banach de $\mathcal{K}(E)\G$ dans $\mathcal{L}_{B\G}(E\G)$, que l'on note $\psi$ par abus de notation et dont l'image est contenue dans $\mathcal{K}(E\G)$.
\end{Pro}
Avant de démontrer cette proposition, démontrons d'abord le lemme suivant
\begin{Lemma}\label{opS}
Pour tout $a\in C_c(G,A)$ et pour tout $g\in G$, les opérateurs $[a,T\G]$, $a(1-(T\G)^2)$ et
$a(g(T\G)-T\G)$ appartiennent à l'image de $\psi$. Plus précisément, pour $a\in C_c(G,A)$ et $g\in G$, posons
\begin{align*}
&S_1:=\big(t\mapsto a(t)(t(T)-T)+[a(t),T]\big),\\
&S_2:=\big(t\mapsto a(t)t(1-T^{2})\big),\\
\hbox{et}\quad&S_3:=\big(t\mapsto a(t)t((gT)-T)\big),
\end{align*}
de sorte que $S_i\in C_c(G,\mathcal{K}(E))$ pour $i=1,..,3$. Alors,
\begin{align*}
&[a,T\G]=\widehat{S_1},\\
&a(1-(T\G)^2)=\widehat{S_2},\\
&a(g(T\G)-T\G)=\widehat{S_3},
\end{align*}
où, pour $i=1,..,3$, $\widehat{S_i}$ est l'élément de $\mathcal{L}_{B\G}(E\G)$ donné à partir de $S_i$ par la définition \ref{Stilda}.
\end{Lemma}
\begin{proof}
Soit $a\in C_c(G,A)$. Pour $x=\int_Gx(g)e_gdg\in C_c(G,E)$,
\begin{align*}
 [a,T\G]^>(x)=
\int_G\big(&\int_Ga(t)t(T^>(x(t^{-1}g)))dt\big)e_gdg\\
&-(T\G)^>\big(g\mapsto\int_Ga(t)t(x(t^{-1}g))dt\big),
\end{align*}
donc, pour tout $g\in G$,
\begin{align*}
[a,T\G]^>(x)(g)&=\int_Ga(t)t(T^>(x(t^{-1}g)))-T^>(a(t)t(x(t^{-1}g)))dt,\\
&=\int_G\big(a(t)(t(T)-T)^>+[a(t),T]^>\big)(t(x(t^{-1}g)))dt.
\end{align*}
De même, pour $\xi=\int_Ge_g\xi(g)dg\in C_c(G,\overline{E})$, on a:
\begin{align*}
[a,T\G]^<(\xi)=((T\G)^<(\xi))a-((T\G)&^<(\xi a)),\\
=\int_Ge_g\int_G(t^{-1}g)^{-1}(T^<(\xi(t))a(t^{-1}&g))dtdg\\
-(T\G)^<(\int_Ge_g\int_G(t^{-1}g)&^{-1}(\xi(t)a(t^{-1}g))dtdg),
\end{align*}
donc, pour tout $g\in G$,
\begin{align*}
[a,T\G]^<(\xi)(g)=\int_G(t^{-1}g)^{-1}&\big(T^<(\xi(t))a(t^{-1}g)\big)dt\\
&-\int_GT^<\big((t^{-1}g)^{-1}(\xi(t)a(t^{-1}g))\big)dt\\
=\int_G(t^{-1}g)^{-1}\Big(T^<(\xi(t))&a(t^{-1}g)-T^<(\xi(t)a(t^{-1}g)),\\
-(t^{-1}g)&T^<(\xi(t)a(t^{-1}g))+T^<(\xi(t)a(t^{-1}g))\Big)dt\\
=\int_G(t^{-1}g)^{-1}\Big([a(t^{-1}g)&,T]^<\xi(t),\\
&+\big(a(t^{-1}g)((t^{-1}g)T-T)\big)^<\xi(t)\Big)dt.
\end{align*}
Donc, si pour tout $t\in G$ on pose $$S_1(t):=a(t)(t(T)-T)+[a(t),T],$$ de sorte que $S_1$ définisse un élément de $C_c(G,\mathcal{K}(E))$, alors $[a,T\G]=\widehat{S_1}$.\\

Calculons maintenant $a(1-T\G^{2})^>$. Soit $x\in C_c(G,E)$, on a:
\begin{align*}
 \big(a(1-T\G^{2})\big)^>(x)&=ax-a(T\G^{2,>})x,
\end{align*}
donc, pour tout $g\in G$,
\begin{align*}
 \big(a(1-T\G^{2})&\big)^>(x)(g),\\
&=\int_Ga(t)t(x(t^{-1}g))dt-\int_G a(t)t\big(T^{2,>}x(t^{-1}g)\big)dt,\\
&=\int_Ga(t)t\big((1-T^{2,>})x(t^{-1}g)\big)dt,\\
&=\int_Ga(t)t(1-T^{2,>})tx(t^{-1}g)dt.
\end{align*}
De même, pour $\xi\in C_c(G,\overline{E})$,
$$(a(1-T\G^{2}))^<(\xi)=\xi a-(T\G^{2,<})(\xi a),$$
d'où, pour tout $g\in G$, on a:
\begin{align*}
 (a(1-T\G^{2}&))^<(\xi)(g)=\xi a(g)-T^{2,<}(\xi a(g)),\\
&=\int_G(t^{-1}g)^{-1}(\xi(t)a(t^{-1}g))-T^{2,<}((t^{-1}g)^{-1}\xi(t)a(t^{-1}g))dt,\\
&=\int_G(t^{-1}g)^{-1}\Big(\xi(t)a(t^{-1}g)-(t^{-1}g)T^{2,<}(\xi(t)a(t^{-1}g))\Big)dt,\\
&=\int_G(t^{-1}g)^{-1}\Big((t^{-1}g)(1-T^{2,<})(\xi(t)a(t^{-1}g))\Big)dt,\\
&=\int_G(t^{-1}g)^{-1}\Big(a(t^{-1}g)(t^{-1}g)(1-T^{2})^<(\xi(t))\Big)dt.
\end{align*}
Si, pour tout $t\in G$, on pose
$$S_2(t):=a(t)t(1-T^{2}),$$
alors $S_2\in C_c(G,\mathcal{K}(E))$ et $a(1-T\G^{2})=\widehat{S_2}$.\\

Prenons maintenant $g\in G$. De la même façon, on peut calculer $a(g(T\G)-T\G)$. Par exemple, pour $x\in C_c(G,E)$,
\begin{align*}
a(g(T\G)-&T\G)^>x(s),\\
&=\int_Ga(t)t(g(T^>)(x(t^{-1}s)))-a(t)tT^>(t(x(t^{-1}s)))dt,\\
&=\int_Ga(t)t((gT^>)-T^>)(t(x(t^{-1}s)))dt,
\end{align*}
et si on pose
$$S_3(t):=a(t)t((gT)-T),$$
alors $S_3\in C_c(G,\mathcal{K}(E))$ et $a(g(T\G)-T\G)=\widehat{S_3}$.
\end{proof}

Nous allons maintenant montrer la proposition \ref{morBancompacts} qui, grâce au lemme \ref{opS}, implique le lemme \ref{bimodtordu}. La démonstration repose sur le lemme suivant analogue au lemme 1.5.6 de \cite{Lafforgue02}
\begin{Lemma}\label{compacts}
Soit $S=(S_g)_{g\in G}\in C_c(G,\mathcal{K}(E))$, où $E$ est vu
comme $B$-paire. Soit
$\widehat{S}\in\mathcal{L}_{B\G}(E\rtimes^{\rho} G)$ l'opérateur défini comme dans la définition \ref{Stilda}. Alors,
\begin{equation}\label{compact}
 \|\widehat{S}\|_{\mathcal{L}_{B\G}(E\rtimes^{\rho} G)}\leq \int_G\|S_g\|_{\mathcal{K}(E)}\|\rho(g)\|_{\mathrm{End}(V)}dg,
\end{equation}
 et $\widehat{S}$ est un opérateur compact. Plus précisément, pour tout $\epsilon>0$, il existe un $n\in\NN$, et pour $i=1,...,n$ il existe des éléments $y_i\in C_c(G,E)$, $\xi_i\in C_c(G,\E)$ tels que, si on pose pour tout $g\in G$,
$$K_g=\int_G\sum\limits_{i=1}^{n}|y_i(t)\rangle\langle g(\xi_i(t^{-1}g))|dt,$$
alors:
\begin{itemize}
 \item  $K=(K_g)_{g\in G}\in C_c(G,\mathcal{K}(E))$,
\item si on considère $y_i$ et $\xi_i$ comme des éléments de $E\rtimes^{\rho}G$ et $\E\rtimes^{\rho}G$ respectivement, on a $\widehat{K}=\sum\limits_{i=1}^{n}|y_i\rangle\langle \xi_i|$,
\item et \begin{equation}\label{epsilon}
\int_G\|S_g-K_g\|_{\mathcal{K}(E)}\|\rho(g)\|_{\mathrm{End}(V)}dg\leq\epsilon.
\end{equation}
\end{itemize}
\end{Lemma}
\begin{proof}
Montrons d'abord que l'inégalité (\ref{compact}) est vraie. Pour ceci, on considère
 l'algèbre
$\left(\begin{array}{cc}
   \mathcal{K}(E)& E\\
\E & B
  \end{array}\right).$ Le produit croisé tordu par $\rho$ de $G$ avec cette algèbre vérifie l'égalité
$$\left(\begin{array}{cc}
   \mathcal{K}(E)& E\\
\E & B
  \end{array}\right)\G=\left(\begin{array}{cc}
\mathcal{K}(E)\G&E\G\\
\E\G& B\G\end{array}\right).$$
\sloppy Ceci implique que $\mathcal{K}(E)\G$ est une sous-algèbre de $\mathcal{L}(E\G)$. Plus précisément,
d'après \cite[Theorem 3.7]{Kasparov88} on a que le produit tensoriel de $C^*$-algèbres $C^*(G,\mathcal{K}(E))\otimes\mathrm{End}(V)$ est une sous-algèbre du produit tensoriel $\mathcal{K}(C^*(G,E))\otimes\mathrm{End}(V)$, et donc elle agit sur le module hilbertien $C^*(G,E)\otimes\mathrm{End}(V)$. D'autre part, par définition $\mathcal{K}(E)\G$ est une sous-algèbre fermée de $C^*(G,\mathcal{K}(E))\otimes\mathrm{End}(V)$ (pour la norme de produit tensoriel de $C^*$-algèbres $\|.\|_ {C^*(G,\mathcal{K}(E))\otimes\mathrm{End}(V)}$) et $E\G$ est un sous-module fermé de $C^*(G,E)\otimes\mathrm{End}(V)$ par construction. Ceci implique que $\mathcal{K}(E)\G$ agit aussi sur $E\G$ et c'est donc une sous-algèbre fermée de $\mathcal{L}(E\G)$, d'où l'égalité
$$\|\widehat{S}\|_{\mathcal{L}(E\G)}=\|S\|_{\mathcal{K}(E)\G}.$$

De plus,
\begin{align*}
\|S\|_{\mathcal{K}(E)\G}&=\|\int_G S_ge_g\otimes\rho(g)dg\|_{C^*(G,\mathcal{K}(E))\otimes\mathrm{End}(V)},\\
&\leq\|\int_G S_ge_g\otimes\rho(g)dg\|_{L^1(G,\mathcal{K}(E)\otimes\mathrm{End}(V))},\\
&\leq\int_G\|S_g\otimes\rho(g)\|_{\mathcal{K}(E)\otimes\mathrm{End}(V)}dg.
\end{align*}
Donc,
$$\|S\|_{\mathcal{K}(E)\G}\leq\int_G\|S_g\|_{\mathcal{K}(E)}\|\rho(g)\|_{\mathrm{End}(V)}dg,$$
d'où l'inégalité (\ref{compact}).\\

Montrons maintenant que $\widehat{S}$ est compact. En utilisant des partitions de l'unité, on voit facilement qu'il suffit de montrer le résultat pour les éléments $S$ dans $C_c(G,\mathcal{K}(E))$ de la forme $S_g=f(g)| y\rangle\langle \xi|$, pour $g\in G$, avec $f\in C_c(G)$, $y\in E$ et $\xi\in\E$.\\
Soit $f$ une fonction à support compact sur $G$. Soit $\epsilon>0$. Il existe une fonction positive $\chi\in C_c(G)$ à support compact contenu dans un voisinage de l'identité $1$ de $G$, telle que $\int_G\chi=1$, et telle que les conditions suivantes soient vérifiées \footnote{Pour que ces conditions soient vérifiées il suffit que le support de $\chi$ soit assez proche de $1$.}
\begin{align*}
&\int_G|f(g)-\chi*f(g)|\,\|\rho(g)\|_{\mathrm{End}(V)}dg\|y\|_E\|\xi\|_{\overline{E}}<\frac{\epsilon}{2}\quad\hbox{et}\\
&\int_G\big(\int_G|\chi(t)f(t^{-1}g)|\,\|\xi-t\xi\|_{E}dt\big)\,\|\rho(g)\|_{\mathrm{End}(V)}dg\|y\|_{E}<\frac{\epsilon}{2}.
\end{align*}
\sloppy Prenons $n=1$ et pour tout $g\in G$, posons $$y_1(g)=\chi(g)y\quad\hbox{et}\quad\xi_1(g)=f(g)g^{-1}(\xi),$$ de sorte que $K_g=\int_G|\chi(t)y\rangle\langle f(t^{-1}g)t(\xi)|dt$ définisse un opérateur $\widehat{K}$ de $\mathcal{K}(E\G)$. On a alors,
\begin{align*}
 \|&S_g-K_g\|_{\scriptstyle\mathcal{K}(E)}=\Big\|f(g)\big|y\big\rangle\big\langle\xi\big|-\int_{G}\chi(t)f(t^{-1}g)\big|y\big\rangle\big\langle t\xi\big|dt\Big\|_{\scriptstyle\mathcal{K}(E)},\\
 &\leq\Big\|f(g)\big|y\big\rangle\big\langle\xi\big|-\int_{G}\chi(t)f(t^{-1}g)\big|y\big\rangle\big\langle \xi\big|dt\Big\|_{\scriptstyle\mathcal{K}(E)}\\
 &\,\,\,\,\,\,\,\,\,\,\,\,\,\,\,\,\,\,\,+\Big\|\int_{G}\chi(t)f(t^{-1}g)\big|y\big\rangle\big\langle \xi\big|dt-\int_{G}\chi(t)f(t^{-1}g)\big|y\big\rangle\big\langle t\xi\big|dt\Big\|_{\scriptstyle\mathcal{K}(E)},\\
 &\leq\Big\|\big(f(g)-\chi*f(g)\big)\big|y\big\rangle\big\langle \xi\big|\Big\|_{\scriptstyle\mathcal{K}(E)}+\int_G\Big\|\,\big|\chi(t)f(t^{-1}g)y\big\rangle\big\langle \xi-t\xi\big|\Big\|_{\scriptstyle\mathcal{K}(E)}\!\!\!\!\!\!\!\!dt\,,\\
 &\leq\big| f(g)-\chi*f(g)\big|\,\big\|y\big\|_{E}\,\big\|\xi\big\|_{\overline{E}}+\int_G\Big|\chi(t)f(t^{-1}g)\Big|\,\big\|y\big\|_E\,\big\| \xi-t\xi\big\|_{\overline{E}}dt.
\end{align*}

On a alors les inégalités suivantes
\begin{align*}
\|\widehat{S}-\widehat{K}&\|_{\scriptstyle\mathcal{L}(E\G)}\leq\int_G\|S_g-K_g\|_{\scriptstyle\mathcal{K}(E)}\|\rho(g)\|_{\scriptstyle\mathrm{End}(V)}dg,\\
&\leq\int_G\Big(\big|f(g)-\chi*f(g)\big|\,\big\|y\big\|_{E}\,\big\|\xi\big\|_{\overline{E}}\\
&\,\,\,\,\,\,\,\,\,\,+\int_G\big|\chi(t)f(t^{-1}g)\big|\,\|y\|_E\,\big\| \xi-t\xi\big\|_{\overline{E}}dt\Big)\|\rho(g)\|_{\scriptstyle\mathrm{End}(V)}dg\\
&< \epsilon.
\end{align*}
D'où l'inégalité \ref{epsilon}.

\end{proof}
En appliquant le lemme \ref{compacts} à $S_1$, $S_2$ et $S_3$ on termine la démonstration du fait que $(E\G,T\G)$ appartient à $E^{\mathrm{ban}}(A\G,B\G)$.
\end{proof}
\begin{Def}\label{descentetordue}
Pour toutes $G$-$C^*$-algèbres $A$ et $B$ et pour toute représentation de dimension finie $\rho$ de $G$, on définit un morphisme de groupes de $KK_G(A,B)$ dans $KK^{\mathrm{ban}}(A\G,B\G)$ (resp. dans $KK^{\mathrm{ban}}(A\rG,B\rG)$) que l'on note $j_{\rho}$ (resp. $j_{\rho,r}$) par la formule suivante: pour $[E,T]\in KK_G(A,B)$, on pose $$j_{\rho}([E,T]):=[E\G,T\G]\quad\hbox{et}\quad j_{\rho,r}([E,T]):=[E\rG,T\rG].$$
On appelle ces morphismes \emph{morphisme de descente tordu} et \emph{morphisme de descente tordu réduit}, respectivement.
\end{Def}

\subsection{Fonctorialité}
La proposition suivante dit que les morphismes de descente tordus sont fonctoriels.
\begin{Pro}\label{descente}
 Les applications
\begin{align*}
j_{\rho}&:KK_G(A,B)\rightarrow KK^{\mathrm{ban}}(A\G,B\G),\\
j_{\rho,r}&:KK_G(A,B)\rightarrow KK^{\mathrm{ban}}(A\rG,B\rG),
\end{align*}
définies dans la définition \ref{descentetordue} sont des morphismes fonctoriels en $A$ et en $B$. De plus, ils sont tels que si $A=B$ alors $j_{\rho}(1_A)=1_{A\G}$ et $j_{\rho,r}(1_A)=1_{A\rG}$.
\end{Pro}
\begin{proof}

On voit facilement que $j_{\rho}(1_A)=1_{A\G}$. Montrons maintenant que $j_{\rho}$ est fonctoriel. La démonstration pour $j_{r,\rho}$ est complètement analogue.\\
Soit $\theta_1$ un morphisme de $G$-$C^*$-algèbres de $A_1$ dans $A$. Notons $\theta_1\rtimes^{\rho} G$ le morphisme d'algèbres de Banach  de $A_1\G$ dans $A\G$ qu'il définit. Il est facile de voir que pour tout $\alpha\in KK_G(A,B)$ on a
$$j_{\rho}(\theta_1^*(\alpha))=(\theta_1\G)^*(j_{\rho}(\alpha)),$$
ce qui donne la fonctorialité en $A$.\\
Soit maintenant $\theta:B\rightarrow C$ un morphisme de $G$-$C^*$-algèbres. On va montrer que pour tout $\alpha\in KK_G(A,B)$,
$$j_{\rho}(\theta_*(\alpha))=(\theta\G)_*(j_{\rho}(\alpha)),$$
dans $KK^{\mathrm{ban}}(A\G,C\G)$.\\
Soit $(E,T)$ un représentant de $\alpha$ dans $E_G(A,B)$. Par définition (cf. \cite[Section 1.1]{Lafforgue02}),
\begin{align*}
(\theta\G)_*(j_{\rho}(E))^>&=E\G\otimes^{\pi}_{\widetilde{B\G}}\widetilde{C\G},\\
(\theta\G)_*(j_{\rho}(E))^<&=\widetilde{C\G}\otimes^{\pi}_{\widetilde{B\G}}\overline{E}\G,
\end{align*}
où $\otimes^{\pi}$ est le produit tensoriel projectif,
et,
\begin{align*}
j_{\rho}(\theta_*(E))^>=(E\otimes_{B}C)\G,\\
j_{\rho}(\theta_*(E))^<=(\overline{E\otimes_{B}C})\G,
\end{align*}
où $\otimes$ est le produit tensoriel interne de modules hilbertiens.\\
On note \begin{align*}
&(\theta\G)_*(j_{\rho}(E)):=(\theta\G)_*(j_{\rho}(E))^>\quad\hbox{et}\quad j_{\rho}(\theta_*(E)):=j_{\rho}(\theta_*(E))^>,\\
 &(\theta\G)_*(j_{\rho}(\overline{E})):=(\theta\G)_*(j_{\rho}(E))^<\quad\hbox{et}\quad j_{\rho}(\theta_*(\overline{E})):=j_{\rho}(\theta_*(E))^<,
\end{align*} pour simplifier les notations.\\
D'après \cite[Lemma 3.10]{Kasparov88}, l'application
\begin{align*}
\tau: C_c(G,E)\otimes C_c(G,C)&\rightarrow C_c(G,E\otimes C)\\
x\otimes c&\mapsto\big(g\mapsto\int_Gx(s)\otimes s c(s^{-1}g)ds\big),
\end{align*}
définit un isomorphisme de $C^*(C,G)$-modules hilbertiens
$$C^*(G,E)\otimes_{C^*(G,B)} C^*(G,C)\rightarrow C^*(G,E\otimes_{B}C).$$

On a alors,
\begin{align*}
 \|\tau(x\otimes c)\|_{(E\otimes_{B}C)\G}&=\|\tau(x\otimes c)\|_{C^*(G,E\otimes_{B}C)\otimes\mathrm{End}(V)}\\
&\leq\|x\|_{C^*(G,E)\otimes\mathrm{End}(V)}\|c\|_{C^*(G,C)\otimes\mathrm{End}(V)}
\end{align*}
ce qui implique que $\|\tau(x\otimes c)\|_{(E\otimes_{B}C)\G}\leq\|x\otimes c\|_{(E\G\otimes^{\pi}_{\widetilde {B\G}}\widetilde{C\G})}$.\\
L'application $\tau$ définit alors un morphisme de $C\G$-modules de Banach à droite de norme inférieure ou égale à $1$:
$$\tau:(E\G)\otimes^{\pi}_{\widetilde{B\G}}(\widetilde{C\G})\rightarrow (E\otimes_{B}C)\G,$$
que l'on note encore $\tau$ par abus de notation. On note $\overline{\tau}$ l'analogue de $\tau$ pour $\overline{E}$.\\
On va alors construire l'homotopie cherchée à l'aide de cônes de la manière suivante (cf. \cite[Section 1.9]{Paravicini06}).\\
Soit $$E\rtimes^{\rho}_{\theta}G=\{(h,x)\in j_{\rho}(\theta_*(E))[0,1]\times (\theta\G)_*(j_{\rho}(E)) | h(0)=\tau(x)\},$$
muni de la norme $\|(h,x)\|=\max\big(\sup\limits_{t\in[0,1]}\|h(t)\|,\|x\|\big)$, le cône associé à $\tau$.\\
De même, on définit
$$\E\rtimes^{\rho}_{\theta}G=\{(h,x)\in j_{\rho}(\theta_*(\E))[0,1]\times(\theta\G)_*(j_{\rho}(\E)) | h(0)=\overline{\tau}(x)\},$$
qui est le cône associé à $\overline{\tau}$.\\
Alors le couple $$(\E\rtimes^{\rho}_{\theta}G,E\rtimes^{\rho}_{\theta}G),$$
définit un $\big(A\G,(C\G)[0,1]\big)$-bimodule de Banach que l'on note $E\rtimes^{\rho}_{\theta}G$ par abus de notation.\\

D'autre part, on définit un opérateur $T\rtimes^{\rho}_{\theta}G\in\mathcal{L}(E\rtimes^{\rho}_{\theta}G)$ de la façon suivante
\begin{align*}
 (T\rtimes^{\rho}_{\theta}G)^>(h,e\otimes c)&=\big(t\mapsto(\theta_*(T)\G)^>h(t),(\theta\G)_*(T\G)^>(x)\big)\\
&=\Big((g,t)\mapsto\theta_*(T)^>(h(t)(g)),(g\mapsto T^>(e(g)))\otimes c\Big),
\end{align*}
pour $(h,e\otimes c)\in E\rtimes^{\rho}_{\theta}G$, et on définit $(T\rtimes^{\rho}_{\theta}G)^<$ de façon analogue.\\
\begin{Rem}
L'opérateur $(T\rtimes^{\rho}_{\theta}G)$ ainsi défini est le ``cône'' du couple d'opérateurs $\big((\theta\rtimes^{\rho}G)_*(T\rtimes^{\rho}G),\theta_*(T)\rtimes^{\rho}G\big)$, noté $$Z\Big((\theta\rtimes^{\rho}G)_*(T\rtimes^{\rho}G),\theta_*(T)\rtimes^{\rho}G\Big)$$ et défini dans \cite[Definition 1.9.14]{Paravicini06}.
\end{Rem}
\begin{Lemma}\label{homotopie}
\sloppy L'élément $(E\rtimes^{\rho}_{\theta}G,T\rtimes^{\rho}_{\theta}G)$ appartient à $E^{\mathrm{ban}}(A\G,(C\G)[0,1])$. De plus, il réalise une homotopie entre $j_{\rho}(\theta_*(\alpha))$ et $(\theta\G)_*(j_{\rho}(\alpha))$.
\end{Lemma}
\begin{proof}

On doit montrer que pour tout $a\in A\rtimes^{\rho} G$ et pour tout $g\in G$, les opérateurs suivants $$[a,(T\rtimes^{\rho}_{\theta}G)],\quad a(1-(T\rtimes^{\rho}_{\theta}G)^2)\quad\text{et}\quad a(g(T\rtimes^{\rho}_{\theta}G)-T\rtimes^{\rho}_{\theta}G)$$
sont des opérateurs compacts de $E\rtimes^{\rho}_{\theta}G$.\\

Soient $a\in A\rtimes^{\rho} G$ et $g\in G$. Pour $S=(S_t)_{t\in G}\in C_c(G,\mathcal{K}(E))$, on note $\theta_*(S):=(\theta_*(S)_t)_{t\in G}$ l'élément de $C_c(G,\mathcal{K}(\theta_*(E))$ tel que $\theta_*(S)_t=\theta_*(S_t)$.\\
On rappelle que, étant donné $S=(S_t)_t\in C_c(G,\mathcal{K}(E))$, on note $\widehat{S}$ l'élément de $\mathcal{L}_{B\G}(E\G)$ associé à $S$ et donné par la définition \ref{Stilda}.\\

On considère l'application suivante
\begin{align*}
\psi:C_c(G,\mathcal{K}(E))&\rightarrow\mathcal{L}(E\rtimes^{\rho}_{\theta}G)\\
S&\mapsto\Big((h,x)\mapsto\big(t\mapsto\widehat{\theta_*(S)}(h(t)),(\theta\G)_*(\widehat{S})x\big)\Big).
\end{align*}

\begin{Lemma}\label{imagecompacte}
L'application $\psi$ induit un morphisme d'algèbre de Banach de $\mathcal{K}(E)\G$ dans $\mathcal{L}(E\rtimes^{\rho}_{\theta}G)$, que l'on note $\psi$ par abus de notation et dont l'image est contenue dans $\mathcal{K}(E\rtimes^{\rho}_{\theta}G)$.
\end{Lemma}
Avant de démontrer le lemme \label{imagecompacte}, remarquons que le lemme suivant implique le lemme \ref{homotopie}
\begin{Lemma}\label{dansimage}
Les opérateurs $$[a,(T\rtimes^{\rho}_{\theta}G)],\quad a(1-(T\rtimes^{\rho}_{\theta}G)^2)\quad\text{et}\quad a(g(T\rtimes^{\rho}_{\theta}G)-T\rtimes^{\rho}_{\theta}G)$$
appartiennent à l'image de $\psi$.
\end{Lemma}
\begin{proof}[Démonstration du lemme \ref{dansimage}]
Soient $S_1,S_2$ et $S_3$ les éléments de $C_c(G,\mathcal{K}(E))$ donnés par le lemme \ref{opS} tels que:
$$\widehat{S_1}=[a,T\G]\!\!\!\quad,\quad\!\!\!\widehat{S_2}=a(1-(T\G)^2)\!\!\!\quad\hbox{et}\quad\!\!\!\widehat{S_3}=a(g(T\G)-T\G).$$

Il est facile de vérifier les égalités suivantes
\begin{align*}
&\psi(S_1)=[a,(T\rtimes^{\rho}_{\theta}G)],\\
&\psi(S_2)=a(1-(T\rtimes^{\rho}_{\theta}G)^2)\\
\hbox{et}\quad&\psi(S_3)=a(g(T\rtimes^{\rho}_{\theta}G)-T\rtimes^{\rho}_{\theta}G).
\end{align*}
En effet, on a par exemple, pour $x\in C_c(G,E)$ et $\xi\in C_c(G,\E)$,
\begin{align*}
 [a,\theta_*(T)\G]^>x(t)&=\int_G(\theta_*(S_{1,s})^>)(s(x(s^{-1})))ds,\\
 [a,\theta_*(T)\G]^<\xi(t)&=\int_G(s^{-1}t)^{-1}(\theta_*(S_{1,s^{-1}t})^<)(\xi(s))ds,
\end{align*}
pour tout $t\in G$, donc, $$[a,\theta_*(T)\G]=\widehat{\theta_*(S_1)}.$$
De même,
\begin{align*}
 [a, (\theta\rtimes^{\rho}G)_*(T\rtimes^{\rho}G)]=(\theta\G)_*(\widehat{S_1}).\\
\end{align*}
Et donc, $[a,(T\rtimes^{\rho}_{\theta}G)]$, $a(1-(T\rtimes^{\rho}_{\theta}G)^2)$ et $a(g(T\rtimes^{\rho}_{\theta}G)-T\rtimes^{\rho}_{\theta}G)$ appartiennent à l'image de $\psi$.
\end{proof}

Le lemme \ref{imagecompacte} implique alors que les opérateurs $$[a,(T\rtimes^{\rho}_{\theta}G)],\quad a(1-(T\rtimes^{\rho}_{\theta}G)^2)\quad\hbox{et}\quad a(g(T\rtimes^{\rho}_{\theta}G)-T\rtimes^{\rho}_{\theta}G)$$ appartiennent à $\mathcal{K}(E\rtimes^{\rho}_{\theta}G)$. Ceci implique que $(E\rtimes^{\rho}_{\theta}G,T\rtimes^{\rho}_{\theta}G)$ appartient à $E^{\mathrm{ban}}(A\G,(C\G)[0,1])$. Il est clair qu'il réalise une homotopie entre $j_{\rho}(\theta_*(\alpha))$ et $(\theta\G)_*(j_{\rho}(\alpha))$ ce qui termine la démonstration du lemme \ref{homotopie}.

\end{proof}

\begin{proof}[Démonstration du lemme \ref{imagecompacte}]

Pour tout $S=(S_t)_t\in C_c(G,\mathcal{K}(E))$, on a
$$\|\psi(S)\|_{\scriptscriptstyle\mathcal{L}(E\rtimes^{\rho}_{\theta}G)}\leq\max\Big(\|\widehat{\theta_*(S)}\|_{\scriptscriptstyle\mathcal{L}(\theta_*(E)\G)},\|(\theta\G)_*(\widehat{S})\|_{\scriptscriptstyle\mathcal{L}((\theta\G)_*(E\G))}\Big).$$
De plus, on a les estimations suivantes
\begin{align*}
 \|(\theta\G)_*(\widehat{S})\|_{\scriptstyle\mathcal{L}((\theta\G)_*(E\G))}&=\|\widehat{S}\otimes 1\|_{\mathcal{L}(E\G\otimes^{\pi}_{\widetilde {B\G}}\widetilde{C\G})},\\
&\leq\|S\|_{\mathcal{K}(E)\G},
\end{align*}
et
\begin{align*}
 \|\widehat{\theta_*(S)}\|_{\mathcal{L}(\theta_*(E)\G)}&=\|\theta_*(S)\|_{\mathcal{K}(\theta_*(E))\G},\\
&\leq\|S\|_{\mathcal{K}(E)\G}.
\end{align*}
Ceci implique que $$\|\psi(S)\|_{\scriptscriptstyle\mathcal{L}(E\rtimes^{\rho}_{\theta}G)}\leq\|S\|_{\mathcal{K}(E)\G},$$et donc que l'application $\psi$ induit bien un morphisme d'algèbre de Banach de $\mathcal{K}(E)\G$ dans $\mathcal{L}(E\rtimes^{\rho}_{\theta}G)$, que l'on note encore $\psi$ par abus de notation. Montrons maintenant que l'image de $\psi$ est contenue dans $\mathcal{K}(E\rtimes^{\rho}_{\theta}G)$. \\
Soient $S\in C_c(G,\mathcal{K}(E))$ et $\epsilon>0$. D'après le lemme \ref{compacts}, il existe $n\in\NN$ et pour $i=1,...,n$, il existe $y_i\in C_c(G,E)$, $\xi_i\in C_c(G,\overline{E})$ tels que l'élément $K=(K_g)_{g\in G}\in C_c(G,\mathcal{K}(E))$ défini par la formule
$$K_g=\int_G\sum\limits_{i=1}^{n}|y_i(t)\rangle\langle g(\xi_i(t^{-1}g))|dt,$$
vérifie l'inégalité
$$\int_G\|S_g-K_g\|_{\mathcal{K}(E)}\|\rho(g)\|_{\mathrm{End}(V)}dg<\epsilon.$$

L'image par $\psi$ de $K$ est un opérateur compact de $E\rtimes^{\rho}_{\theta}G$. En effet,
pour tout $s\in G$, $\theta_*(K_s)=\int_G\sum\limits_{i=1}^{n}|y_i(t)\otimes 1\rangle\langle s(1\otimes\xi_i(t^{-1}s)|dt$, et donc, pour tout $x\in C_c(G,\theta_*(E)^>)$,
\begin{align*}
 \widehat{\theta_*(K)}^>(x)(g)&=\int_G\theta_*(K_s)^>s(x(s^{-1}g))ds,\\
&=\int_G\int_G\sum\limits_{i=1}^{n}|y_i(t)\otimes 1\rangle\langle s(1\otimes\xi_i(t^{-1}s))|^>s(x(s^{-1}g))dtds,\\
&=\sum\limits_{i=1}^{n}\big(|g\mapsto y_i(g)\otimes 1\rangle\langle g\mapsto1\otimes\xi_i(g)|^>(x)\big)(g).
\end{align*}
De façon analogue, pour $\xi\in C_c(G,\theta_*(E)^<)$, on trouve
\begin{align*}
\widehat{\theta_*(K)}^<(\xi)(g)&=\int_G(s^{-1}g)^{-1}\theta_*(K_{s^{-1}g})^<(\xi(s))ds,\\
&=\sum\limits_{i=1}^{n}\big(|g\mapsto y_i(g)\otimes 1\rangle\langle g\mapsto1\otimes\xi_i(g)|^<(\xi)\big)(g);
\end{align*}
d'où, $$\widehat{\theta_*(K)}=\sum\limits_{i=1}^{n}|g\mapsto y_i(g)\otimes 1\rangle\langle g\mapsto1\otimes\xi_i(g)|.$$
De plus,
\begin{align*}
(\theta\G)_*(\widehat{K})&=(1\otimes\widehat{K}^<,\widehat{K}^>\otimes 1),\\
&=\sum\limits_{i=1}^n|y_i\otimes1\rangle\langle1\otimes\xi_i|,
\end{align*}
et donc pour $(h,x)\in E\rtimes_{\theta}^{\rho}G$,
\begin{align*}
&\psi(K)^>(h,x)=\\
&\Big(t\!\mapsto\!\!\sum\limits_{i=1}^{n}|g\mapsto y_i(g)\otimes 1\rangle\langle g\mapsto1\otimes\xi_i(g)|^>(h(t)), \sum\limits_{i=1}^n|y_i\otimes1\rangle\langle1\otimes\xi_i|^>x\Big),\\
&\,\,=\sum\limits_{i=1}^{n}\Big|\big(t\mapsto (g\mapsto y_i(g)\otimes 1),y_i\otimes 1\big)\Big\rangle\Big\langle\big(t\mapsto (g\mapsto 1\otimes\xi_i(g)),1\otimes\xi_i\big)\Big|.
\end{align*}
On en déduit que $\psi(K)\in \mathcal{K}(E\rtimes_{\theta}^{\rho}G)$.\\

Le fait que $\psi$ soit un morphisme d'algèbres de Banach, implique alors que $\psi(S)$ est un opérateur compact de $E\rtimes^{\rho}_{\theta}G$. En effet, on a les inégalités suivantes:
\begin{align*}
\|\psi(S)-\psi(K)\|_{\mathcal{L}(E\rtimes^{\rho}_{\theta}G)}&\leq\|S-K\|_{\mathcal{K}(E)\G},\\
&\leq\int_G\|S_g-K_g\|_{\mathcal{K}(E)}\|\rho(g)\|_{\mathrm{End}(V)}<\epsilon.
\end{align*}
\end{proof}

Ceci termine la démonstration de la proposition \ref{descente} et donc de la fonctorialité des morphismes de descente tordus.
\end{proof}

\subsection{Descente et action de $KK^{\mathrm{ban}}$ sur la $K$-théorie.}
Soient $A$ et $B$ deux algèbres de Banach. On rappelle que dans \cite{Lafforgue02}, Lafforgue à construit un morphisme de groupes $$\Sigma:KK^{\mathrm{ban}}(A,B)\rightarrow \mathrm{Hom}(K(A),K(B)),$$qui induit une action de la $KK$-théorie banachique sur la $K$-théorie. La proposition suivante montre que les morphismes de descente tordus sont compatibles avec $\Sigma$.

\begin{Pro}\label{compatib}
\sloppy Soient $G$ un groupe localement compact, $A$, $B$, $C$ des $G$-$C^*$-algèbres,  $\alpha\in KK_G(A,B)$, $\beta\in KK_G(B,C)$, et $\alpha\otimes_B\beta$ leur produit de Kasparov qui est un élément de $KK_G(A,B)$ (cf. \cite{Kasparov88}). On a alors,
$$\Sigma(j_{\rho}(\alpha\otimes_B\beta))=\Sigma(j_{\rho}(\beta))\circ\Sigma(j_{\rho}(\alpha))$$
dans $\mathrm{Hom}(K(A\G),K(B\G))$. De même dans le cas du produit croisé réduit.
\end{Pro}
\begin{proof}
 La démonstration découle de la fonctorialité de $j_{\rho}$ (resp. de $j_{\rho,r}$) et du lemme suivant démontré dans \cite[Proposition 1.6.10]{Lafforgue02} qui dit que tout élément de $KK$-théorie de Kasparov est le produit de Kasparov d'un élément qui provient d'un morphisme et d'un élément qui est l'inverse d'un morphisme.

\begin{Lemma}\label{morphisme}
 Soient $G$ un groupe localement compact, $A$ et $B$ deux $G$-$C^*$-algèbres et $\alpha\in KK_G(A,B)$. Il existe une $G$-$C^*$-algèbre $A_1$, des morphisme $G$-équivariants $\theta:A_1\rightarrow A$, $\eta:A_1\rightarrow B$, et un élément $\alpha_1\in KK_G(A,A_1)$ tels que:\\
\begin{enumerate}
\item$[\theta]\otimes_A\alpha_1=\mathrm{Id}_{A_1}$ et $\alpha_1\otimes_{A}[\theta]=\mathrm{Id}_A$, où $[\theta]\in KK_G(A_1,A)$ est induit par le morphisme $\theta$ (c'est-à-dire que $\alpha_1$ est l'inverse en $KK$-théorie $G$-équivariante d'un morphisme),   \\
et \item $\theta^{*}(\alpha)=[\eta]$, où $[\eta]\in KK_G(A_1,A)$ est l'élément induit par le morphisme $\eta$.
\end{enumerate}

\end{Lemma}
La fonctorialité du morphisme de descente $j_{\rho}$ et de l'action de $KK^{\mathrm{ban}}$ sur la K-théorie donnée par $\Sigma$ impliquent la proposition (\ref{compatib}). La démonstration est la même que celle de \cite[Proposition 1.6.9]{Lafforgue02}:
on applique le lemme (\ref{morphisme}) à $G, A, B$ et $\alpha$. Comme $j_{\rho}$ et $\Sigma$ sont fonctoriels on a
\begin{align*}
 \Sigma(j_{\rho}(\alpha_1))\circ j_{\rho}(\theta)_*&=\Sigma(j_{\rho}(\theta)^*(j_{\rho}(\alpha_1))),\\
&=\Sigma(j_{\rho}(\theta^*(\alpha_1))),\\
&=\mathrm{Id}_{K(A_1\G)},\,\,\,\hbox{car}\,\,\,\theta^*(\alpha_1)=\mathrm{Id}_{A_1}.
\end{align*}
De même,
\begin{align*}
 j_{\rho}(\theta)_*\circ\Sigma(j_{\rho}(\alpha_1))&=\Sigma(j_{\rho}(\theta)_*(j_{\rho}(\alpha_1)),\\
&=\Sigma(j_{\rho}(\theta_*(\alpha_1))),\\
&=\mathrm{Id}_{K(A\G)}\,\,\,\hbox{car}\,\,\,\theta_*(\alpha_1)=\mathrm{Id}_A.
\end{align*}
\sloppy Donc, $$j_{\rho}(\theta)_*:K(A_1\G)\rightarrow K(A\G)$$ est inversible. De plus, $\theta^*(\alpha\otimes_B \beta)=\eta^*(\beta)$ dans $KK_G(A_1,C)$ et donc,
$$\Sigma(j_{\rho}(\alpha\otimes\beta))\circ j_{\rho}(\theta)_*=\Sigma(j_{\rho}(\beta))\circ j_{\rho}(\eta)_*,$$
et ceci implique que
$$\Sigma(j_{\rho}(\alpha\otimes\beta))=\Sigma(j_{\rho}(\beta))\circ j_{\rho}(\eta)_*\circ j_{\rho}(\theta)_*^{-1}.$$
De la même façon, on montre, en prenant $C=B$ et $\beta=\mathrm{Id}$, que $$\Sigma(j_{\rho}(\alpha))=j_{\rho}(\eta)_*\circ j_{\rho}(\theta)_*^{-1}$$ Ceci implique la proposition.\\
La même démonstration vaut pour la descente dans le cas du produit croisé réduit.

\end{proof}

\subsection{Construction du morphisme tordu}
\sloppy Soit $G$ un groupe localement compact et $B$ une $G$-$C^*$-algèbre. On rappelle que l'on note $\underline{E}G$
 le classifiant universel pour les actions propres de $G$ et $K^{\mathrm{top}}(G,B)$ la $K$-homologie $G$-équivariante de $\underline{E}G$ à valeurs dans $B$. On rappelle que $K^{\mathrm{top}}(G,B)$ est donné par la formule suivante
$$K^{\mathrm{top}}(G,B)=\lim\limits_{\longrightarrow}KK_G(C_0(X),B),$$où la limite inductive est prise parmi les parties fermées $X$ de $\underline{E}G$ qui sont $G$-invariantes et $G$-compactes.

\begin{Def}\label{Mischenko}
Soit $X$ une partie fermée $G$-compacte de $\underline{E}G$. Soit $c$ une fonction continue à support compact sur $X$ et à valeurs dans $\RR_+$
telle que $\int_Gc(g^{-1}x)dg=1$, pour tout $x\in X$ (une fonction avec ces propriétés existe d'après \cite{Tu99} et elle est appelée ``fonction de cut-off'' sur $X$). Soit $p$ la fonction sur $G\times X$ définie par la formule $$p(g,x)=\sqrt{c(x)c(g^{-1}x)},$$
pour $(g,x)\in G\times X$.\\
La fonction $p$ définit alors un projecteur de $C_c(G,C_0(X))$, que l'on note $p$ par abus de notation. L'élément de $K(C_0(X)\G)$ qu'il définit est appelé \emph{élément de Mischenko tordu} associé à $X$. On le note $\Delta_{\rho}$.\\
Soit $$\Sigma: KK^{\mathrm{ban}}(C_0(X)\G,B\G)\rightarrow
\mathrm{Hom}(K( C_0(X)\G),K(B\G))$$ le morphisme provenant de
l'action de
 $KK^{\mathrm{ban}}$ sur la $K$-théorie défini dans \cite{Lafforgue02}. On a alors une suite de morphismes
$$KK_G(C_0(X),B)\stackrel{j_{\rho}}{\rightarrow}KK^{\mathrm{ban}}(C_0(X)\G,B\G)\stackrel{\Sigma(.)(\Delta_{\rho})}
{\longrightarrow}K(B\G).$$
De même que dans \cite{Lafforgue02} section 1.7, en passant à la limite inductive on définit un morphisme
$$\mu_{\rho}^{B}:K^{\mathrm{top}}(G,B)\rightarrow K(B\G),$$
et on l'appelle \emph{morphisme de Baum-Connes tordu par la représentation $\rho$}.
\end{Def}
\begin{Rem}
La fonctorialité des morphismes $\Sigma$ et $j_{\rho}$ implique que le morphisme de Baum-Connes tordu par $\rho$ est fonctoriel en $B$. En effet, soit $\theta:B\rightarrow C$ un morphisme de $G$-$C^*$-algèbres. Soit $\alpha\in KK_G(C_0(X),B)$. On a les égalités suivantes,
\begin{align*}
(\theta\G)_*(\mu_{\rho}^{B}(\alpha))&=(\theta\G)_*(\Sigma(j_{\rho}(\alpha))(\Delta_{\rho})),\\
&=\Sigma\big((\theta\G)_*(j_{\rho}(\alpha))\big)(\Delta_{\rho}),\\
&=\Sigma\big(j_{\rho}(\theta_*(\alpha))\big)(\Delta_{\rho}),\\
&=\mu_{\rho}^C(\theta_*(\alpha)).
\end{align*}
\end{Rem}
\begin{Def}
 Pour toute $G$-$C^*$-algèbre $B$, soit $\lambda_{G,B}^\rho$ le morphisme d'algèbres de Banach de $B\G$ dans $B\rG$ qui prolonge l'identité sur $C_c(G,B)$. Soit $$(\lambda_{G,B}^{\rho})_*:K(B\G)\rightarrow K(B\rG)$$ le morphisme induit par $\lambda_{G,B}^{\rho}$ en $K$-théorie. On définit un \emph{morphisme de Baum-Connes tordu réduit},
$$\mu_{\rho,r}^B:K^{\mathrm{top}}(G,B)\rightarrow K(B\rG),$$
en posant $\mu_{\rho,r}^B:=(\lambda_{G,B}^{\rho})_*\circ\mu_{\rho}^B$.
\end{Def}
\begin{Rem}\label{casréduit1}
 Il est facile de voir que, si $X$ est une partie $G$-compacte de $\underline{E}G$ et si on note $\Delta_{\rho,r}$ l'élément de $K(C_0(X)\rG)$ défini par la fonction $p$, alors le morphisme de Baum-Connes tordu réduit vérifie l'égalité $$\mu_{\rho,r}^B(x)=\Sigma\big(j_{\rho,r}(x)\big)(\Delta_{\rho,r}),$$
pour tout $x\in KK_G(C_0(X),B)$.
\end{Rem}

\subsection{Compatibilité avec la somme directe de représentations}
On va maintenant montrer que le morphisme de Baum-Connes tordu est compatible avec la somme directe de représentations. On utilisera ce résultat dans l'étude de la bijectivité.
\begin{Lemma}\label{commutatifs}
 Soient $\rho$ et $\rho'$ deux représentations de dimension finie de $G$ et $B$ une $G$-$C^*$-algèbre. Alors il existe des morphismes
  $$i_{\rho'}:B\rtimes^{\rho\oplus\rho'}G\rightarrow B\rtimes^{\rho'}G\quad\hbox{et}\quad i_{\rho}:B\rtimes^{\rho\oplus\rho'}G\rightarrow B\G,$$
  qui prolongent l'identité sur $C_c(G,B)$ et tels que les diagrammes suivants,

$$\xymatrix{
    K^{\mathrm{top}}(G,B)\ar[dr]_{\mu_{\rho'}^B}\ar[r]^{\mu_{\rho\oplus\rho'}^B} & K(B\rtimes^{\rho\oplus\rho'}G)\ar[d]^{i_{\rho',*}}\\
     & K(B\rtimes^{\rho'}G),\\}\quad\hbox{et}\quad
\xymatrix{
    K^{\mathrm{top}}(G,B)\ar[dr]_{\mu_{\rho}^B}\ar[r]^{\mu_{\rho\oplus\rho'}^B} & K(B\rtimes^{\rho\oplus\rho'}G)\ar[d]^{i_{\rho,*}}\\
     & K(B\G),\\}$$
soient commutatifs. On a le même résultat dans le cas des produits croisés tordus réduits.
\end{Lemma}
\begin{proof}
 Il est clair que pour toute $G$-$C^*$-algèbre $A$ et pour toute fonction $f\in C_c(G,A)$,
$$\|f\|_{A\G}\leq\|f\|_{A\rtimes^{\rho\oplus\rho'}G},$$
donc $\mathrm{Id}_{C_c(G,A)}$ s'étend en un morphisme d'algèbres de Banach $$i_{\rho}:A\rtimes^{\rho\oplus\rho'}G\rightarrow A\G.$$
En fait, on a que $A\rtimes^{\rho\oplus\rho'}G=A\G\cap A\rtimes^{\rho'}G$, à équivalence de norme près.\\
 De même, si $E$ est un $A$-module de Banach et $f\in C_c(G,E)$,
$$\|f\|_{E\G}\leq\|f\|_{E\rtimes^{\rho\oplus\rho'}G},$$ et $\mathrm{Id}_{C_c(G,E)}$ s'étend en un morphisme de $A$-modules de Banach, que l'on note aussi $i_{\rho}$, de $E\rtimes^{\rho\oplus\rho'}G$ dans $E\G$.\\
 On doit montrer que $$\mu^B_{\rho}=i_{\rho,*}\circ\mu^B_{\rho\oplus\rho'}.$$
Pour ceci, on va d'abord montrer que, pour tout $\alpha\in KK_G(A,B)$, les éléments $j_{\rho}(\alpha)$ et $j_{\rho\oplus\rho'}(\alpha)$ ont la même image dans le groupe $KK^{\mathrm{ban}}(A\rtimes^{\rho\oplus\rho'}G,B\G)$. C'est le lemme suivant
\begin{Lemma}\label{homotopiedescente}
Pour tout $\alpha\in KK_G(A,B)$, on a l'égalité, $$i^*_{\rho}(j_{\rho}(\alpha))=i_{\rho,*}(j_{\rho\oplus\rho'}(\alpha)),$$
dans $KK^{\mathrm{ban}}(A\rtimes^{\rho\oplus\rho'}G,B\G)$.
\end{Lemma}
Avant de démontrer le lemme \ref{homotopiedescente}, remarquons qu'il implique le lemme \ref{commutatifs}. En effet, pour $X$ une partie $G$-compacte de $\underline{E}G$ et pour un élément $\alpha$ dans $KK_G(C_0(X),B)$, $$\mu^B_{\rho\oplus\rho'}(\alpha)=\Sigma(j_{\rho\oplus\rho'}(\alpha))(\Delta_{\rho\oplus\rho'}),$$ donc, on a les égalités suivantes,
\begin{align*}
i_{\rho,*}(\mu^B_{\rho\oplus\rho'}(\alpha))&=i_{\rho,*}\circ\Sigma(j_{\rho\oplus\rho'}(\alpha))(\Delta_{\rho\oplus\rho'}),\\
&=\Sigma(i_{\rho,*}(j_{\rho\oplus\rho'}(\alpha)))(\Delta_{\rho\oplus\rho'}),\\
&=\Sigma(i_{\rho}^*(j_{\rho}(\alpha)))(\Delta_{\rho\oplus\rho'}),\\
&=\Sigma(j_{\rho}(\alpha))(i_{\rho,*}(\Delta_{\rho\oplus\rho'})),\\
&=\Sigma(j_{\rho}(\alpha))(\Delta_{\rho})=\mu^B_{\rho}(\alpha).
\end{align*}
On a alors que $\mu_{\rho}^B=i_{\rho,*}\circ\mu^B_{\rho\oplus\rho'}$, pour toute $G$-$C^*$-algèbre $B$, ce qui termine la démonstration du lemme \ref{commutatifs}.

\end{proof}

\begin{proof}[Démonstration du lemme \ref{homotopiedescente}]
Soit $(E,T)$ un représentant de $\alpha$ dans $E_G(A,B)$. Alors
$$i^*_{\rho}(j_{\rho}(E))=E\G,$$ où $E\G$ est un $(A\rtimes^{\rho\oplus\rho'}G,B\G)$-bimodule de Banach, l'action de $A\rtimes^{\rho\oplus\rho'}G$ étant donnée par $i_{\rho}$, et
\begin{align*}
i_{\rho,*}(j_{\rho\oplus\rho'}(E))^>&=(E\rtimes^{\rho\oplus\rho'}G)^>\otimes^{\pi}_{\widetilde{B\rtimes^{\rho\oplus\rho'}G}}\widetilde{(B\G)},\\
i_{\rho,*}(j_{\rho\oplus\rho'}(E))^<&=\widetilde{(B\G)}\otimes^{\pi}_{\widetilde{B\rtimes^{\rho\oplus\rho'}G}}(E\rtimes^{\rho\oplus\rho'}G)^<.                                                                                                               \end{align*}

On considère l'application suivante,
\begin{align*}
\tau:C_c(G,E)\otimes_{C_c(G,B)}C_c(G,B)&\rightarrow C_c(G,E)\\
x\otimes b&\mapsto xb.
\end{align*}
On a alors,
\begin{align*}
 \|\tau(x\otimes b)\|_{E\G}&\leq \|x\|_{E\G}\|b\|_{B\G},\\
&\leq\|x\|_{E\rtimes^{\rho\oplus\rho'}G}\|b\|_{B\G},
\end{align*}
et donc $\tau$ définit une application,
\begin{align*}
(E\rtimes^{\rho\oplus\rho'}G)\otimes^{\pi}_{\widetilde{B\rtimes^{\rho\oplus\rho'}G}}\widetilde{B\G}&\rightarrow E\G,
\end{align*}
que l'on note encore $\tau$ par abus de notation.\\
Comme $$\|\tau(x\otimes b)\|_{E\G}\leq\|x\otimes b\|_{(E\rtimes^{\rho\oplus\rho'}G)\otimes^{\pi}_{\widetilde{B\rtimes^{\rho\oplus\rho'}G}}(\widetilde{B\G})},$$ l'application $\tau$ définit un morphisme de $B\G$-modules de Banach à droite de norme inférieure ou égale à $1$.\\
De même, il existe un morphisme de $B\G$-modules de Banach (à gauche)
$$\overline{\tau}:\widetilde{B\G}\otimes^{\pi}_{\widetilde{B\rtimes^{\rho\oplus\rho'}G}} (\E\rtimes^{\rho\oplus\rho'}G)\rightarrow \E\G,$$
de norme inférieure ou égale à $1$.\\
On va construire une homotopie entre $(i_{\rho,*}(j_{\rho\oplus\rho'}(\alpha))$ et $j_{\rho}(\alpha)$ en utilisant les cônes associés à ces morphismes.

Soit,
$$\mathcal{C}(\tau)^>=\{(h,x)\in (E\G)[0,1]\times i_{\rho,*}(j_{\rho\oplus\rho'}(E))^> | h(0)=\tau(x)\},$$
le cône associé à $\tau$, et
$$\mathcal{C}(\tau)^<=\{(h,x)\in (\overline{E}\G)[0,1]\times i_{\rho,*}(j_{\rho\oplus\rho'}(E))^< | h(0)=\overline{\tau}(x)\},$$
le cône associé à $\overline{\tau}$.\\
On pose $$\mathcal{C}(\tau):=(\mathcal{C}(\tau)^<,\mathcal{C}(\tau)^>),$$ qui est un $\big(A\rtimes^{\rho\oplus\rho'}G,(B\G)[0,1]\big)$-bimodule de Banach.

Soit $\mathcal{C}(\tau, T)\in\mathcal{L}(\mathcal{C}(\tau))$, l'opérateur sur $\mathcal{C}(\tau)$ défini par la formule suivante
\begin{align*}
\mathcal{C}(\tau,T)^>(h,e\otimes b)=&\Big(t\mapsto(T\G)(h(t)),\big((T\rtimes^{\rho\oplus\rho'}G)^>e\otimes b\big)\Big),
\end{align*}
pour $(h,e\otimes b)\in\mathcal{C}(\tau)^>$ et $\mathcal{C}(\tau,T)^<$ défini de façon analogue sur $\mathcal{C}(\tau)^>$. On a alors le lemme suivant

\begin{Lemma}\label{presquehomot}
 L'élément $(\mathcal{C}(\tau),\mathcal{C}(\tau,T))$ défini ci-dessus est un élément de $E^{\mathrm{ban}}\big(A\rtimes^{\rho\oplus\rho'}G,(B\G)[0,1]\big)$.
\end{Lemma}

\begin{proof}
\sloppy On doit montrer que pour tout $a\in C_c(G,A)$ et pour tout $g\in G$, les opérateurs,
$$[a,\mathcal{C}(\tau,T)],\quad a(1-(\mathcal{C}(\tau,T))^2)\quad\hbox{et}\quad a(g(\mathcal{C}(\tau,T))-\mathcal{C}(\tau,T))$$ sont des opérateurs compacts de $\mathcal{C}(\tau)$.

Soit $S=(S_g)_{g\in G}\in C_c(G,\mathcal{K}(E))$ et soit $\widehat{S}_{\rho}$ (resp. $\widehat{S}_{\rho\oplus\rho'}$) l'élément de $\mathcal{L}_{B\G}(E\G)$ (resp. de $\mathcal{L}_{B^{\rho\oplus\rho'}\rtimes G}(E^{\rho\oplus\rho'}\rtimes G)$) associé à $S$ par la définition \ref{Stilda} appliquée à la représentation $\rho$ (resp. $\rho\oplus\rho'$).
On a alors une application $$\psi:C_c(G,\mathcal{K}(E))\rightarrow\mathcal{L}(\mathcal{C}(\tau))$$ définie par la formule
$$\psi(S)^>(h,e\otimes b)=\Big(t\mapsto\widehat{S}^>_{\rho}(h(t)),\widehat{S}^>_{\rho\oplus\rho'}(e)\otimes b\Big),$$
où $S\in C_c(G,\mathcal{K}(E))$ et $(h,e\otimes b)\in\mathcal{C}(\tau)^>$,
et où $\psi(S)^<$ est défini sur $\mathcal{C}(\tau)^<$ de façon analogue.

\begin{Lemma}\label{imagecompacte2}
 L'application $\psi$ induit un morphisme d'algèbres de Banach de $\mathcal{K}(E)\rtimes^{\rho\oplus\rho'}G$ dans $\mathcal{L}(\mathcal{C}(\tau))$, que l'on note $\psi$ par abus de notation et dont l'image est contenue dans $\mathcal{K}(\mathcal{C}(\tau))$.
\end{Lemma}
Avant de démontrer le lemme \label{imagecompacte2}, remarquons que le lemme suivant implique le lemme \ref{presquehomot}
\begin{Lemma}\label{dansimage2}
 Les opérateurs $$[a,\mathcal{C}(\tau,T)],\quad a(1-(\mathcal{C}(\tau,T))^2)\quad\text{et}\quad a(g(\mathcal{C}(\tau,T))-\mathcal{C}(\tau,T))$$
appartiennent à l'image de $\psi$.
\end{Lemma}

\begin{proof}[Démonstration du lemme \ref{dansimage2}]
Si on pose, pour tout $a\in C_c(G,A)$, pour tout $g\in G$ et pour tout $g_1\in G$,
\begin{align*}
&S_1(g_1):=a(g_1)(g_1(T)-T)+[a(g_1),T]),&\\
&S_2(g_1):=a(g_1)g_1(1-T^{2})),&\\
\hbox{et}\quad&S_3(g_1):=a(g_1)g_1((gT)-T)),&
\end{align*}
de sorte que,
\begin{align*}
&\widehat{S_1}=[a,T\G],&\,\,\,\,\,\,\,\,\,\,\,\,\,\,\,\,\,\,&\\
&\widehat{S_2}=a(1-(T\G)^2),&&\\
\hbox{et}\quad&\widehat{S_3}=a(g(T\G)-T\G),&&
\end{align*}
 (voir lemme \ref{opS}), alors, par des calculs simples, on obtient,
\begin{align*}
&\psi(S_1)=[a,\mathcal{C}(\tau,T)],\\
&\psi(S_2)=a(1-(\mathcal{C}(\tau,T))^2),\\
\hbox{et}\quad&\psi(S_3)=a(g(\mathcal{C}(\tau,T))-\mathcal{C}(\tau,T)).
\end{align*}
\end{proof}
Il est clair que les lemmes \ref{imagecompacte2} et \ref{dansimage2} impliquent le lemme \ref{presquehomot} car pour tout $a\in C_c(G,A)$ et pour tout $g\in G$, les opérateurs $[a,\mathcal{C}(\tau,T)]$, $a(1-(\mathcal{C}(\tau,T))^2)$ et $a(g(\mathcal{C}(\tau,T))-\mathcal{C}(\tau,T))$ sont contenus dans l'image de $\psi$ qui, d'après le lemme \ref{imagecompacte2}, est contenue dans l'algèbre des opérateurs compacts de $\mathcal{C}(\tau)$. Ceci implique donc que $(\mathcal{C}(\tau),\mathcal{C}(\tau,T))$ appartient à $E^{\mathrm{ban}}\big(A\rtimes^{\rho\oplus\rho'}G,(B\G)[0,1]\big)$, ce qui termine la démonstration du lemme \ref{presquehomot}.
\end{proof}
Il est clair que $(\mathcal{C}(\tau),\mathcal{C}(\tau,T))$ réalise alors une homotopie entre $i^*_{\rho}(j_{\rho}(\alpha))$ et $i_{\rho,*}(j_{\rho\oplus\rho'}(\alpha))$, et ceci termine la démonstration du lemme \ref{homotopiedescente}.
\end{proof}

\begin{proof}[Démonstration du lemme \ref{imagecompacte2}]
Soit $$(h,e\otimes b)\in C_c(G,E)\times (C_c(G,E)\otimes C_c(G,B)).$$ Alors,
\begin{align*}
\|\psi(S)^>(h,e&\otimes b)\|_{\mathcal{C}(\tau)^>}\\
&=\max\Big(\sup\limits_{t\in [0,1]}\|\widehat{S}^>_{\rho}(h(t))\|_{E\G},\|\widehat{S}^>_{\rho\oplus\rho'}(e)\otimes b\|_{i_{\rho,*}(j_{\rho\oplus\rho'}(E))^>}\Big),
\end{align*}
et donc,
\begin{align*}
\|\psi(S)\|_{\mathcal{L}(\mathcal{C}(\tau))}&\leq\max\Big(\|\widehat{S}_{\rho}\|_{\mathcal{L}(E\G)},\|\widehat{S}_{\rho\oplus\rho'}\|_{\mathcal{L}(E\rtimes^{\rho\oplus\rho'}G)}\Big), \\
&\leq \int_G\|S_g\|\|(\rho\oplus\rho')(g)\|_{\mathrm{End(V\oplus V')}}dg.
\end{align*}
L'application $\psi$ définit alors un morphisme d'algèbres de Banach de l'algèbre $\mathcal{K}(E)\rtimes^{\rho\oplus\rho'}G$ dans $\mathcal{L}(\mathcal{C}(\tau))$, car $$\|\widehat{S}_{\rho\oplus\rho'}\|_{\mathcal{L}(E\rtimes^{\rho\oplus\rho'}G)}=\|S\|_{\mathcal{K}(E)\rtimes^{\rho\oplus\rho'}G},$$ (voir la démonstration du lemme \ref{compacts}).\\
On va maintenant montrer que l'image de $\psi$ est contenue dans l'algèbre des opérateurs compacts de $\mathcal{C}(\tau)$. Soit $S=(S_g)_{g\in G}\in C_c(G,\mathcal{K}(E))$ et soit $\epsilon>0$. D'après le lemme \ref{compacts}, il existe $n\in\NN$ et pour $i=1,..,n$, il existe des éléments $y_i\in C_c(G,E)$ et $\xi_i\in C_c(G,\overline{E})$ tels que, l'élément $K=(K_g)_{g\in G}$ de $C_c(G,\mathcal{K}(E))$ définit par la formule,
$$K_g=\int_G\sum\limits_{i=1}^{n}|y_i(g_1)\rangle\langle g(\xi_i(g_1^{-1}g))|dt,$$
vérifie l'inégalité,
$$\int_G\|S_g-K_g\|_{\mathcal{K}(E)}\|(\rho\oplus\rho')(g)\|_{\mathrm{End}(V\oplus V')}dg<\epsilon.$$
Ceci implique
\begin{align*}
\|\psi(S)-\psi(K)\|_{\mathcal{L}(\mathcal{C}(\tau))}&\leq \int_G\|S_g-K_g\|\|(\rho\oplus\rho')(g)\|_{\mathrm{End}(V\oplus V')}dg,\\
&\leq\epsilon.
\end{align*}
Mais, $\psi(K)$ appartient à $\mathcal{K}(\mathcal{C}(\tau))$, car
\begin{align*}
 \psi(K)=&\Big(\sum\limits^{n}_{i=1}|t\mapsto y_i\rangle\langle t\mapsto\xi_i|,\sum\limits^n_{i=1}|g\mapsto y_i(g)\otimes 1\rangle\langle g\mapsto 1\otimes\xi_i(g)|\Big),\\
=& \sum\limits^{n}_{i=1}\Big|\big(t\mapsto y_i,g\mapsto y_i(g)\otimes 1\big)\Big\rangle\Big\langle \big(t\mapsto\xi_i, g\mapsto 1\otimes\xi_i(g)\big)\Big|,
\end{align*}
et donc $\psi(S)$ est approché par des sommes finies d'opérateurs de rang $1$ ce qui implique que $\psi(S)$ appartient à $\mathcal{K}(\mathcal{C}(\tau))$.
\end{proof}


\section{Groupes admettant un élément $\gamma$ de Kasparov}\label{dirac}
Nous allons maintenant montrer que les morphismes de Baum-Connes tordus, maximal et réduit, sont des isomorphismes pour tout groupe $G$ admettant un élément $\gamma$ de Kasparov égal à $1$ dans $KK_G(\CC,\CC)$. Pour ceci, nous allons d'abord montrer que les morphismes tordus à coefficients dans une algèbre propre sont toujours des isomorphismes.
\subsection{Coefficients dans une algèbre propre}
On rappelle qu'une $G$-$C^*$-algèbre $B$ est propre s'il existe un $G$-espace propre $Z$ tel que $B$ soit une $C_0(Z)$-$G$-$C^{*}$-algèbre au sens de Kasparov \cite[1.5]{Kasparov88} (c'est-à-dire qu'il existe un morphisme $\Theta$
de $C_0(Z)$ dans le centre de l'algèbre $M(B)$ des
multiplicateurs de $B$, $G$-équivariant et tel que $\Theta(C_0(Z))B=B$). Par abus de notation, si $B$ est une $G$-$C^*$-algèbre propre on identifie $C_0(Z)$ à son image dans le centre de $M(B)$. De façon équivalente, $B$ est une $G$-$C^*$-algèbre propre si et seulement si, il existe un
$G$-espace propre $Z$ tel que $B$ soit munie d'une action du
groupoïde $Z\rtimes G$. On renvoie le lecteur à \cite{LeGall97} pour la définition de l'action d'un groupoïde sur une $C^*$-algèbre. On rappelle tout de même, que si $Z$ est un $G$-espace, on note $Z\rtimes G$ le groupoïde dont l'ensemble des unités est $Z$, l'ensemble des flèches est le produit cartésien $Z\times G$ et les applications source et but sont données, respectivement, par les applications suivantes
$$s:(z,g)\mapsto g.z\quad\hbox{et}\quad r:(z,g)\mapsto z.$$
On note alors $(Z\rtimes G)^{(2)}:=(Z\rtimes G)^{(1)}\times_{(Z\rtimes G)^{(0)}}(Z\rtimes
G)^{(1)}$ l'ensemble des éléments composables de $Z\rtimes G$.\\
Si $B$ est une $C_0(Z)$-$G$-$C^{*}$-algèbre, on note $r^{*}(B)$ la $C_0(Z\rtimes G)$-$G$-$C^{*}$-algèbre obtenue comme image réciproque de $B$ par l'application $r$.\\
On rappelle aussi que si $B$ est une
$G$-$C^*$-algèbre propre, alors $$C^*(G,B)=C^*_r(G,B)$$ (cf. \cite[page 184]{Kasparov-Skandalis03}, \cite[page 192]{Higson-Guentner}). Nous allons montrer le théorème suivant
\begin{Theo}\label{propre}
 Si $B$ est une $G$-$C^*$-algèbre propre, alors $\mu_{\rho}^{B}$ est un isomorphisme.
\end{Theo}
\begin{Rem}\label{casréduit}
 On remarque que si $B$ est une $G$-$C^*$-algèbre propre, alors $$B\G=B\rG,$$car $C^*(G,B)=C^*_r(G,B)$, donc le théorème \ref{propre} implique que le morphisme de Baum-Connes tordu réduit à coefficients dans une algèbre propre est aussi un isomorphisme.
\end{Rem}

\begin{proof}[Démonstration du théorème \ref{propre}]
On a le lemme suivant
\begin{Lemma}\label{K-th}
 Si $B$ est une $G$-$C^*$-algèbre propre, les morphismes $$i_{\rho}:B\rtimes^{\rho\oplus 1_G}G\rightarrow B\rtimes^{\rho}G\quad\hbox{et}\quad i_{1_G}:B\rtimes^{\rho\oplus 1_G}G\rightarrow C^*(G,B),$$où l'on note $1_G$ la représentation triviale de $G$, induisent des isomorphismes en $K$-théorie.
\end{Lemma}

Le lemme \label{K-th} implique le théorème \ref{propre}. En effet, d'après le lemme
\ref{commutatifs}, $i_{1_G,*}\circ\mu_{\rho\oplus 1_G}^B=\mu^B$,
où $\mu_B$ est un isomorphisme car c'est le morphisme de
Baum-Connes usuel à valeurs dans une algèbre propre \cite{Chabert-Echterhoff-Meyer}. On en déduit que $\mu^B_{\rho\oplus 1_G}$ est
un isomorphisme. Comme d'autre part,
$i_{\rho,*}\circ\mu_{\rho\oplus 1_G}^B=\mu^B_{\rho}$ et que
$i_{\rho,*}$ est aussi un isomorphisme, on en déduit que le
morphisme de Baum-Connes tordu par $\rho$ à coefficients dans une
algèbre propre, $\mu^B_{\rho}$, est un isomorphisme.

\end{proof}

Pour montrer le lemme \ref{K-th}, on va utiliser un résultat de Lafforgue (cf. \cite[Lemme 1.7.8]{Lafforgue02}. On rappelle qu'une sous-algèbre $D$ d'une algèbre $A$ est dite héréditaire dans $A$ si $DAD\subset D$. Pour nous, l'intérêt de cette notion est donnée par le lemme suivant démontré dans  \cite[Lemme 1.7.10]{Lafforgue02}
\begin{Lemma}
 Si $C$ est une algèbre de Banach, $B$ est une sous-algèbre de Banach dense de $C$ et s'il existe une sous-algèbre dense de $B$ qui est héréditaire dans $C$, alors $B$ et $C$ ont la même $K$-théorie.
\end{Lemma}
De plus, Lafforgue a montré le lemme suivant (cf. \cite[Lemme 1.7.8]{Lafforgue02})
\begin{Lemma}
Soit $Z$ un $G$-espace propre tel que $Z\rtimes G$ agisse sur $B$
et soient $s,r:Z\rtimes G\rightarrow Z$ les applications source et but, respectivement, de $Z\rtimes
G$. Soit $B_c$ la sous-algèbre de $B$ formée des éléments $b$ de
$B$ tels que $fb=b$ pour un certain $f\in C_c(Z)$. Alors
$D=C_c(G,B_c)$, l'algèbre des sections continues à support compact dans $Z\rtimes G$
de $r^*(B)$, est une sous-algèbre héréditaire de $C^*(G,B)$.
\end{Lemma}
Nous allons donner ici la démonstration de Lafforgue avec plus de détails par souci de commodité pour le lecteur.
\begin{proof}
On rappelle que l'on note tout élément $f$ de $C_c(G,B)$ par l'intégrale formelle $\int_Gf(g)e_gdg$ et que $dg^{-1}=\Delta(g^{-1})dg$. Soient $f_1,f_3$ des éléments de $C_c(G,B)$. L'application
\begin{align*}
C_c(G,B)&\rightarrow C_c(G,B)\\
f_2&\mapsto f_1*f_2*f_3,
\end{align*}
se prolonge par continuité en une application de $C^*_r(G,B)$ dans l'espace des fonctions $f$ continues sur $G$ à valeurs dans $B$ qui vérifient la condition suivante: il existe une constante $C$ telle que pour tout $g\in G$,
$$\|f(g)\|_B\Delta(g)^{\frac{1}{2}}<C.$$ En effet, soit $\mathrm{L}^2(G,B)$ muni de la structure de $B$-module hilbertien donnée par la formule: $$\langle f,f'\rangle_{B}=\int_Gf(t)^*f'(t)dt,$$
pour tout $f,f'\in\mathrm{L}^2(G,B)$. On note tout élément de $L^2(G,B)$ par l'intégrale formelle $\int_Ge_gf(g)dg$ de sorte que $B$ agisse à droite sur $L^2(G,B)$. Avec ces conventions, l'application de $C_c(G,B)$ dans $L^2(G,B)$ envoie $\int_Gf(g)e_gdg$ dans $\int_Ge_gf(g)dg$, donc il faut faire attention avec les formules.\\
Si $f=\int_Ge_gf(g)dg$, on a $f^*=\int_Gf(g)^*e_{g^{-1}}dg$ et donc
\begin{align*}
(f{^*}*f')&=(\int_Gf(g)^*e_{g^{-1}}dg)(\int_Ge_gf'(g)dg),\\
&=\int_{G\times G}f(g)^*tf'(gt)dge_tdt.
\end{align*}
Si on note $1$ l'identité de $G$, ceci implique que $f^**f'(1)=\langle f,f'\rangle_{B}$. Donc, pour tout $f,f'\in C_c(G,B)$ et pour tout $g\in G$,
\begin{align*}
\|f*f'(g)\|_B&=\|f*(f'e_{g^{-1}})(1)\|_B,\\
&=\|\langle f^*,(f'e_{g^{-1}})\rangle_{\mathrm{L}^2(G,B)}\|_B,\\
&\leq\|f^*\|_{\mathrm{L}^2(G,B)}\|f'e_{g^{-1}}\|_{\mathrm{L}^2(G,B)}.
\end{align*}
Or,
\begin{align*}
\|f'e_{g^{-1}}\|_{\mathrm{L}^2(G,B)}&=\Big\|\int_Gf'(tg)^*f'(tg)dt\Big\|_B^{\frac{1}{2}},\\
&=\Big\|\int_Gf'(t)^*f(t)\Delta(g^{-1})dt\Big\|_B^{\frac{1}{2}},\\
&\leq\|f'\|_{L^2(G,B)}\Delta(g)^{-\frac{1}{2}}.
\end{align*}
Ceci implique que, pour $f_1,f_3,f_2\in C_c(G,B)$ et $g\in G$, on a les inégalités suivantes
\begin{align*}
\|f_1*f_2*f_3(g)\|_B&\leq\Delta(g)^{-\frac{1}{2}}\|f_1^*\|_{L^2(G,B)}\|f_2*f_3\|_{L^2(G,B)},\\
&\leq\Delta(g)^{-\frac{1}{2}}\|f_1^*\|_{L^2(G,B)}\|\lambda_{(G,B)}(f_2)f_3\|_{L^2(G,B)},\\
&\leq\Delta(g)^{-\frac{1}{2}}\|f_1^*\|_{L^2(G,B)}\|f_2\|_{C^*_r(G,B)}\|f_3\|_{L^{2}(G,B)},
\end{align*}
où on note $\lambda_{G,B}$ la représentation régulière de $C_r^*(G,B)$ dans $\mathcal{L}(\mathrm{L}^2(G,B))$ de sorte que $\|\lambda_{(G,B)}(f_2)\|_{\mathcal{L}(L^{2}(G,B))}=\|f_2\|_{C^*_r(G,B)}$. On a donc
$$\|f_1*f_2*f_3(g)\|_B\Delta(g)^{\frac{1}{2}}\leq\|f_1^*\|_{L^2(G,B)}\|f_2\|_{C^*_r(G,B)}\|f_3\|_{L^{2}(G,B)},$$
ce qu'on voulait démontrer.\\

Maintenant, si $f_1,f_3$ appartiennent à $C_c(G,B_c)$, ce sont des sections
continues à support compact sur $Z\rtimes G$ de $r^*(B)$, vu
comme champ continu d'algèbres au-dessus de $Z\rtimes G$; donc $r(\mathrm{supp}_{Z\rtimes G}(f_1))$ et $s(\mathrm{supp}_{Z\rtimes G}(f_3))$ sont des sous-ensembles compacts de $Z\rtimes G$. De plus, comme $Z$ est un espace $G$-propre, l'application $$(r,s):Z\rtimes G\rightarrow Z\times Z,$$ est propre et donc le sous-ensemble de $Z\rtimes G$
$$K:=\{(z,h)| r(z,h)\in r(\mathrm{supp}_{Z\rtimes G}(f_1)),s(z,h)\in s(\mathrm{supp}_{Z\rtimes G}(f_3))\}$$ est compact.\\
Soit $\phi$ l'application qui à deux éléments composables de $Z\rtimes G$ associe leur composée:
\begin{align*}
\phi:(Z\rtimes G)^{(2)}&\rightarrow(Z\rtimes G)^{(1)}\\
\{((z',h),(z,g))|z'=gz\}&\mapsto(z,hg).
\end{align*}
Si on définit un produit $*$ entre les parties de $Z\rtimes G$ de la
manière suivante: pour $X, Y\subset Z\rtimes G$,
$$X*Y:=\phi(X\times_{(Z\rtimes G)^{(0)}}Y),$$ alors le support du produit d'éléments de $C_c(G,B)$ est contenue dans le produit des supports. On a alors que, pour $f_2\in C_c(G,B)$,
$$\mathrm{supp}_{Z\rtimes
G}(f_1*f_2*f_3)\subset\mathrm{supp}_{Z\rtimes
G}(f_1)*\mathrm{supp}_{Z\rtimes G}(f_2)*\mathrm{supp}_{Z\rtimes
G}(f_3).$$
Or, $\mathrm{supp}_{Z\rtimes
G}(f_1)*\mathrm{supp}_{Z\rtimes G}(f_2)*\mathrm{supp}_{Z\rtimes
G}(f_3)\subset K$, car:
\begin{align*}
\mathrm{supp}_{Z\rtimes
G}(f_1)*\mathrm{supp}&_{Z\rtimes G}(f_2)*\mathrm{supp}_{Z\rtimes
G}(f_3)\\
&=\{(z,g)\in Z\rtimes
G|g=g_1g_2g_3\,\text{avec\,}\,g_1,g_2,g_3\in G,\\
&\,\,\,\,\,\,\,\,\,\,\,(z,g_3)\in\mathrm{supp}_{Z\rtimes
G}(f_3),
(g_3z,g_2)\in\mathrm{supp}_{Z\rtimes G}(f_2),\\
&\,\,\,\,\,\,\,\,\,\,\,(g_2g_3z,g_1)\in\mathrm{supp}_{Z\rtimes
G}(f_1)\}.
\end{align*}
donc le support de $f_1*f_2*f_3$ est inclus dans un sous-ensemble compact de $Z\rtimes G$ qui ne dépend que de $f_1$ et $f_3$. On en déduit que pour $f_1,f_3\in C_c(G,B_c)$ l'application $f_2\mapsto f_1*f_2*f_3$ a pour image $C_c(G,B_c)$, car sur le support de $f_1*f_2*f_3$ la fonction $g\mapsto\Delta(g)$ est alors bornée. Ceci implique que $C_c(G,B_c)$ est une sous-algèbre héréditaire de $C_r^*(G,B)$.
Comme, de plus, $C_r^*(G,B)$ est égal à $C^*(G,B)$ car $B$ est
propre, on en déduit que $C_c(G,B_c)$ est une algèbre héréditaire
de $C^*(G,B)$.
\end{proof}

\begin{proof}[Démonstration du lemme \ref{K-th}]
En gardant les notations du lemme précédent, l'algèbre
$D$ est une sous-algèbre dense de $B\rtimes^{\rho\oplus 1_G}G$ car
$B_c$ est dense dans $B=C_0(Z)B$. De plus, comme $D$ est une
sous-algèbre héréditaire de $C^*(G,B)$, alors
$D\otimes\mathrm{End}(V)$ est héréditaire dans
$C^*(G,B)\otimes\mathrm{End}(V)$ et comme
$D\otimes\mathrm{End}(V)\cap B\rtimes^{\rho}G=D$ alors $D$ est une
sous-algèbre héréditaire de $B\rtimes^{\rho}G$.\\

\sloppy En appliquant les lemmes \cite[Lemme 1.7.9 et Lemme
1.7.10]{Lafforgue02}, on obtient alors que
\begin{align*}
&i_{\rho,*}:K(B\rtimes^{\rho\oplus 1_G}G)\rightarrow K(B\G),\\
\hbox{et}\quad& i_{1_G,*}:K(B\rtimes^{\rho\oplus 1_G}G)\rightarrow K(C^*(G,B))
\end{align*}
sont des isomorphismes.
\end{proof}

\begin{Rem}
D'après le lemme précédent, si $B$ est une $G$-$C^*$-algèbre propre, l'application
$i_{\rho,*}\circ i_{1_G,*}^{-1}$ de $K(C^*(G,B))$ dans $K(B\G)$ est un isomorphisme pour tout groupe localement compact.
\end{Rem}

\subsection{{\'E}lément $\gamma$ de Kasparov}
On va maintenant utiliser le résultat obtenu pour les algèbres propres pour montrer que le morphisme de Baum-Connes tordu par n'importe quelle représentation de dimension finie de $G$ et à coefficient dans une $G$-$C^*$-algèbre quelconque, est un isomorphisme pour tout groupe localement compact $G$ qui admet un élément $\gamma$ de Kasparov égal à $1$ dans $KK_G(\CC,\CC)$. Le résultat est donné par le théorème suivant

\begin{Theo}\label{gamma=1}
\sloppy Soit $G$ un groupe localement compact tel que il existe une $G$-$C^*$-algèbre propre $A$, et des éléments $\eta\in KK_G(\CC,A)$ et $d\in KK_G(A,\CC)$
 tels que, si on pose $\gamma:=\eta\otimes_A d\in KK_G(\CC,\CC)$ on a $\gamma=1$. Soit $B$ une $G$-$C^*$-algèbre. Alors, pour toute
 représentation $\rho$ de dimension finie de $G$, $\mu_{\rho}^B$ et $\mu^B_{\rho,r}$ sont des isomorphismes.

\end{Theo}
\begin{proof}
Si $A,B,D$ sont des $G$-$C^*$-algèbres, on note $\sigma_D$ le
morphisme de $KK_G(A,B)$ dans $KK_G(A\otimes D,B\otimes D)$ défini
dans \cite{Kasparov88}.\\
Soient $A$, $\eta$ et $d$ vérifiant les hypothèses du théorème.
L'injectivité et la surjectivité découlent de
la commutativité du diagramme suivant
\setcounter{equation}{0}
\begin{equation}\label{BC}
\xymatrix{
    K^{\mathrm{top}}(G,B)\ar[d]_{\mu^B_{\rho}}\ar[rr]^{\sigma_B(\eta)_*} &&K^{\mathrm{top}}(G,A\otimes B)\ar[rr]^{\sigma_B(d)_*}\ar[d]^{\mu^{B\otimes A}_{\rho}}&&K^{\mathrm{top}}(G,B)\ar[d]_{\mu^B_{\rho}}\\
K(B\G)\ar[rr]_{\Sigma(j_{\rho}(\sigma_B(\eta)))}&&K(A\otimes B\G)\ar[rr]_{\Sigma(j_{\rho}(\sigma_B(d)))}&&K(B\G).\\
}
\end{equation}
On va démontrer d'abord la surjectivité. Supposons qu'il existe
$\gamma\in KK_G(\CC,\CC)$ vérifiant les hypothèses et tel que
$\gamma=1$ dans $KK_G(\CC,\CC)$.\\

Le fait que $\gamma$ soit égal à $1$ implique que
$\Sigma(j_{\rho}(\sigma_B(\gamma)))=\mathrm{Id}_{K(B\G)}$. Or,
$\gamma=\eta\otimes_A d$, donc
$\sigma_B(\gamma)=\sigma_B(\eta)\otimes_A\sigma_B(d)$ et donc
\begin{align*}
\Sigma(j_{\rho}(\sigma_B(\gamma)))&=\Sigma(j_{\rho}(\sigma_B(\eta)\otimes_{A\otimes B}\sigma_B(d))),\\
&=\Sigma(j_{\rho}(\sigma_B(d)))\circ\Sigma(j_{\rho}(\sigma_B(\eta))),
\end{align*}
ce qui implique que $\Sigma(j_{\rho}(\sigma_B(d)))$ est
surjectif.\\
D'autre part,
$$\Sigma(j_{\rho}(\sigma_B(d)))\circ\mu_{\rho}^{A\otimes
B}=\mu^{B}_{\rho}\circ\sigma_{B}(d)_*,$$ et comme $A\otimes
B$ est une algèbre propre car $A$ est propre,
$\mu_{\rho}^{A\otimes B}$ est un isomorphisme et ceci implique que
$\mu_{\rho}^B$ est surjectif.\\

\sloppy Montrons maintenant l'injectivité.
Soit $x\in K^{\mathrm{top}}(G,B)$ tel
que $\mu_{\rho}^B(x)=0$. On a alors
\begin{align*}
 \mu_{\rho}^{A\otimes B}(\sigma_B(\eta)_*(x))&=\Sigma(j_{\rho}(\sigma_B(\eta)))(\mu^B_{\rho}(x)),\\
&=0,
\end{align*}
ce qui implique que $\sigma_B(\eta)_*(x)=0$ car
$\mu_{\rho}^{A\otimes B}$ est un isomorphisme. Mais
$\sigma_B(\gamma)=\sigma_B(\eta)\otimes_{A\otimes B}\sigma_B(d)$,
donc $\sigma_B(\gamma)_*(x)=0$. De plus, le fait que $\gamma=1$
implique que $\sigma_B(\gamma)=1$ et donc que
$\sigma_B(\gamma)_*=\mathrm{Id}_{K^{\mathrm{top}}(G,B)}$. Ceci implique que
$x=0$.\\
Les remarques \ref{casréduit1} et \ref{casréduit} impliquent que la même démonstration est valable dans le cas du morphisme tordu réduit.

\end{proof}
Plus généralement, le diagramme (\ref{BC}) permet aussi de montrer que l'existence d'un élément $\gamma$ de Kasparov implique l'injectivité du morphisme de Baum-Connes tordu. C'est le théorème suivant
\begin{Theo}
\sloppy Supposons que pour toute partie $G$-compacte $Y$ de $\underline{E}G$, il existe une $G$-$C^*$-algèbre propre $A$ et des éléments $\eta\in KK_G(\CC,A)$ et $d\in KK_G(A,\CC)$ tels que $\gamma=\eta\otimes_A d\in KK_G(\CC,\CC)$ vérifie $p^*(\gamma)=1$ dans $KK_{G\ltimes Y}(C_0(Y),C_0(Y))$, où $p$ est la projection de $Y$ vers le point. Alors, pour toute représentation $\rho$ et pour toute $G$-$C^*$-algèbre $B$, les morphismes $\mu_{\rho}^B$ et $\mu^B_{\rho,r}$ sont sont injectifs.
\end{Theo}
\begin{proof}
 Soit $x$ un élément de $K^{\mathrm{top}}(G)$ tel que $\mu_{\rho}^B(x)=0$. Soit $Y$ une partie $G$-compacte de $\underline{E}G$ telle que $x\in KK_G(C_0(Y),B)$ et soient $A$, $\eta$, $d$ et $\gamma$ vérifiant les hypothèses du théorème. On va montrer que $x=0$. La commutativité du diagramme (\ref{BC}) implique que $\sigma_B(\eta)_*(x)=0$ car
\begin{align*}
 \mu_{\rho}^{A\otimes B}(\sigma_B(\eta)_*(x))&=\Sigma(j_{\rho}(\sigma_B(\eta)))(\mu^B_{\rho}(x)),\\
&=0,
\end{align*}
et $\mu_{\rho}^{A\otimes B}$ est un isomorphisme (car $A\otimes B$ est une algèbre propre). Mais $\sigma_B(\eta)_*(x)=0$ implique que $\sigma_B(\gamma)_*(x)=0$, car $\sigma_B(\gamma)=\sigma_B(\eta)\otimes_{A\otimes B}\sigma_B(d)$. \\
 D'autre part, l'égalité $p^*(\gamma)=1$ dans $KK_{G\ltimes Y}(C_0(Y),C_0(Y))$ implique que $\sigma_{C_0(Y)}(\gamma)^{*}x=x$. Or, comme $\gamma\in KK_G(\CC,\CC)$, on a
\begin{align*}
\sigma_{C_0(Y)}(\gamma)\otimes_{C_0(Y)}x&=x\otimes_{B}\sigma_B(\gamma).
\end{align*}
Ceci implique que $\sigma_B(\gamma)_*x=x$ et donc que $x=0$. La même démonstration est valable dans le cas du morphisme tordu réduit.
\end{proof}

\bibliographystyle{amsalpha}
\bibliography{bibliographie}

\def\cprime{$'$} \def\cprime{$'$}
\providecommand{\bysame}{\leavevmode\hbox to3em{\hrulefill}\thinspace}
\providecommand{\MR}{\relax\ifhmode\unskip\space\fi MR }
\providecommand{\MRhref}[2]{%
  \href{http://www.ams.org/mathscinet-getitem?mr=#1}{#2}
}
\providecommand{\href}[2]{#2}
\begin{thebibliography}{CEM01}

\bibitem[BCH94]{Baum-Connes-Higson}
P.~Baum, A.~Connes, and N.~Higson, \emph{Classifying space for proper actions
  and {$K$}-theory of group {$C\sp \ast$}-algebras}, $C\sp \ast$-algebras:
  1943--1993 (San Antonio, TX, 1993), Contemp. Math., vol. 167, Amer. Math.
  Soc., Providence, RI, 1994, pp.~240--291.

\bibitem[Bos90]{Bost90}
J.~B. Bost, \emph{Principe d'{O}ka, {$K$}-th\'eorie et systèmes dynamiques non
  commutatifs}, Invent. Math. \textbf{101} (1990), 261--333.

\bibitem[CEM01]{Chabert-Echterhoff-Meyer}
J.~Chabert, S.~Echterhoff, and R.~Meyer, \emph{Deux remarques sur l'application
  de {B}aum-{C}onnes}, C. R. Acad. Sci. Paris S\'er. I Math. \textbf{332}
  (2001), no.~7, 607--610.

\bibitem[Cha03]{Chatterji03}
I.~Chatterji, \emph{Property ({RD}) for cocompact lattices in a finite product
  of rank one {L}ie groups with some rank two {L}ie groups}, Geom. Dedicata
  \textbf{96} (2003), 161--177.

\bibitem[GA07a]{GomezThese}
M.~P. Gomez-Aparicio, \emph{Propriéte {$(T)$} tordue et morphisme de
  {B}aum-{C}onnes tordus par une représentation non-unitaire}, Ph.D. thesis,
  Université de Paris VII, décembre 2007.

\bibitem[GA07b]{Gomez07}
\bysame, \emph{Sur la propriété {$(T)$} tordue par un produit tensoriel}, J.
  Lie Theory \textbf{17} (2007), 505--524.

\bibitem[GA08]{Gomez08-2}
\bysame, \emph{Représentations non-unitaires, morphisme de {B}aum-{C}onnes
  tordu et complétions inconditionnelles}, preprint, january 2008.

\bibitem[HG04]{Higson-Guentner}
N.~Higson and E.~Guentner, \emph{Group {$C\sp \ast$}-algebras and
  {$K$}-theory}, Noncommutative geometry, Lecture Notes in Math., vol. 1831,
  Springer, Berlin, 2004, pp.~137--251.

\bibitem[HK01]{Higson-Kasparov}
N.~Higson and G.~G. Kasparov, \emph{{$E$}-theory and {$KK$}-theory for groups
  which act properly and isometrically on {H}ilbert space}, Invent. Math.
  \textbf{144} (2001), no.~1, 23--74.

\bibitem[HLS02]{Higson-Lafforgue-Skandalis}
N.~Higson, V.~Lafforgue, and G.~Skandalis, \emph{Counterexamples to the
  {B}aum-{C}onnes conjecture}, Geom. Funct. Anal. \textbf{12} (2002), no.~2,
  330--354.

\bibitem[JK95]{Julg-Kasparov}
P.~Julg and G.~Kasparov, \emph{Operator {$K$}-theory for the group
  ${SU}(n,1)$}, J. Reine Angew. Math. \textbf{463} (1995), 99--152 (English).

\bibitem[Jul97]{Julg97}
P.~Julg, \emph{Remarks on the {B}aum-{C}onnes conjecture and {K}azhdan's
  property {$T$}}, Operator algebras and their applications (Waterloo, ON,
  1994/1995), Fields Inst. Commun., vol.~13, Amer. Math. Soc., Providence, RI,
  1997, pp.~145--153.

\bibitem[Kas88]{Kasparov88}
G.~G. Kasparov, \emph{Equivariant {$KK$}-theory and the {N}ovikov conjecture},
  Invent. Math. \textbf{91} (1988), no.~1, 147--201.

\bibitem[Kaz67]{Kazhdan}
D.~A. Kazhdan, \emph{On the connection of the dual space of a group with the
  structure of its closed subgroups}, Funkcional. Anal. i Prilo\v zen.
  \textbf{1} (1967), 71--74.

\bibitem[KS91]{Kasparov-Skandalis91}
G.~G. Kasparov and G.~Skandalis, \emph{Groups acting on buildings, operator
  {$K$}-theory, and {N}ovikov's conjecture}, $K$-Theory \textbf{4} (1991),
  no.~4, 303--337.

\bibitem[KS03]{Kasparov-Skandalis03}
\bysame, \emph{Groups acting properly on ``bolic'' spaces and the {N}ovikov
  conjecture}, Ann. of Math. (2) \textbf{158} (2003), no.~1, 165--206.

\bibitem[Laf00]{Lafforgue00}
V.~Lafforgue, \emph{A proof of property ({RD}) for cocompact lattices of {${\rm
  SL}(3,\bold R)$} and {${\rm SL}(3,\bold C)$}}, J. Lie Theory \textbf{10}
  (2000), no.~2, 255--267.

\bibitem[Laf02a]{LafforgueICM02}
\bysame, \emph{Banach {$KK$}-theory and the {B}aum-{C}onnes conjecture},
  Proceedings of the International Congress of Mathematicians, Vol. II
  (Beijing, 2002) (Beijing), Higher Ed. Press, 2002, pp.~795--812.

\bibitem[Laf02b]{Lafforgue02}
\bysame, \emph{{$K$}-th\'eorie bivariante pour les alg\`ebres de {B}anach et
  conjecture de {B}aum-{C}onnes}, Invent. Math. \textbf{149} (2002), no.~1,
  1--95.

\bibitem[LG97]{LeGall97}
P.~Y. Le~Gall, \emph{Th\'eorie de {K}asparov \'equivariante et groupo\"\i des},
  C. R. Acad. Sci. Paris S\'er. I Math. \textbf{324} (1997), no.~6, 695--698.

\bibitem[Par06]{Paravicini06}
W.~Paravicini, \emph{{$KK$}-theory for {B}anach algebras and proper groupoids},
  Ph.D. thesis, november 2006.

\bibitem[Ska91]{Skandalis91}
G.~Skandalis, \emph{Kasparov's bivariant {$K$}-theory and applications}, Expo.
  Math. \textbf{9} (1991), no.~3, 193--250.

\bibitem[Tu99]{Tu99}
J.~L. Tu, \emph{La conjecture de {N}ovikov pour les feuilletages
  hyperboliques}, $K$-Theory \textbf{16} (1999), no.~2, 129--184.

\bibitem[Tza00]{Tzanev00}
K.~Tzanev, \emph{{$C^*$}-algèbres de {H}ecke et {${\rm K}$}-théorie}, Ph.D.
  thesis, Université de Paris VII, 2000.

\end{thebibliography}

\end{document}